\documentclass[11pt, reqno]{amsart}
\usepackage[T1]{fontenc}
\usepackage{amsfonts}
\usepackage{amsmath}
\usepackage{amsthm}
\usepackage{amssymb}
\usepackage[textwidth=16cm, textheight=21.1cm, hmarginratio=1:1, vmarginratio=1:1]{geometry}
\usepackage{graphicx}
\usepackage{xcolor}
\usepackage{mathtools}
\usepackage[colorlinks=true, linkcolor=blue, citecolor=green]{hyperref}
\usepackage{cleveref}
\usepackage[english]{babel}
\usepackage{lineno}
\usepackage{float}
\usepackage{stmaryrd}

\newtheorem{theorem}{Theorem}[section]
\newtheorem{definition}[theorem]{Definition}

\newtheorem{lemma}[theorem]{Lemma}
\newtheorem{conjecture}[theorem]{Conjecture}

\usepackage{tikz}
\usetikzlibrary{positioning}
\usetikzlibrary{decorations,arrows}
\usetikzlibrary{decorations.markings}
\numberwithin{equation}{section}

\newcommand{\ints}[1]{\llbracket #1 \rrbracket}
\newcommand{\intss}[2]{\llbracket #1,#2 \rrbracket}

\usepackage{rustic}
\usepackage[T1]{fontenc}

\emergencystretch=\maxdimen
\title{Proper edge coloring with rainbow diamonds}
\author{Runze Wang}
\address[]{Department of Mathematics and Computer Science, Augustana College, Rock Island, IL 61201, USA}
\email{runze.w@hotmail.com}
\thanks{}

\begin{document}

\sloppy

\begin{abstract}
    Motivated by the B-coloring defined by Gy\'arf\'as and S\'ark\"ozy, we introduce a new edge coloring called \emph{D-coloring}. For a graph $G$, a D-coloring of $G$ is a proper edge coloring such that every diamond subgraph is rainbow. The \emph{D-chromatic index} of $G$, denoted by $\chi'_D(G)$, is the minimum number of colors needed for a D-coloring of $G$. Denote by $\Delta$ the maximum degree of $G$. We propose a Brooks-type conjecture stating that $\chi'_D(G)\le \frac{1}{2}\Delta^2+\frac{1}{2}\Delta$ and verify this conjecture for $\Delta\le 5$. We also prove a general upper bound $\chi'_D(G)\le \frac{9}{16}\Delta^2+\frac{1}{2}\Delta$.
\end{abstract}

\maketitle

\section{Introduction}
In this paper, we only study finite simple graphs. For a graph $G=(V,\,E)$ and two edges $e,\,e'\in E$, the distance between $e$ and $e'$, denoted by $d(e,\,e')$, is the distance between the corresponding vertices in the line graph of $G$. So, if two edges are incident, then the distance between them is $1$; if two edges are not incident but connected by a third edge, then the distance between them is $2$. An edge coloring of $G$ is called a \emph{strong edge coloring} if for any distinct $e,\,e'\in E$ with $d(e,\,e')\le 2$, $e$ and $e'$ get different colors. The \emph{strong chromatic index} of $G$, denoted by $\chi'_s(G)$, is the minimum number of colors needed for a strong edge coloring of $G$. Regarding the strong chromatic indices of graphs with restricted maximum degree, Erd\H os and Ne\v set\v ril~\cite{EN} proposed the following conjecture.

\begin{conjecture}[Erd\H os and Ne\v set\v ril~\cite{EN}]\label{conj1}
For every graph $G$ with maximum degree $\Delta$,
\[
    \chi'_s(G)\le\begin{cases}
        \frac{5}{4}\Delta^2 &\text{if } \Delta \text{ is even}, \\
        \frac{5}{4}\Delta^2-\frac{1}{2}\Delta+\frac{1}{4} &\text{if } \Delta \text{ is odd}.
    \end{cases}
\]
\end{conjecture}

Denote by $P_n$ the path with $n$ vertices, and by $C_n$ the cycle with $n$ vertices. If Conjecture~\ref{conj1} holds, the conjectured upper bound will be sharp, as it can be attained by a balanced blowup of $C_5$.

Regarding this conjecture, the first non-trivial case $\Delta=3$ has been verified by Andersen~\cite{An} and by Hor\'ak et al.~\cite{HHT}. While the conjecture is still open for $\Delta\ge 4$, partial progress has been made in the case $\Delta=4$ (see~\cite{HSY,LLY}) and the case $\Delta=5$ (see~\cite{LLH}). Regarding general upper bounds, Molloy and Reed~\cite{MR} proved that $\chi'_s(G)\le 1.998\Delta^2$ for sufficiently large $\Delta$. The coefficient was subsequently reduced to $1.93$ by Bruhn and Joos~\cite{BJ}, to $1.835$ by Bonamy et al.~\cite{BPP}, to $1.772$ by Hurley et al.~\cite{HDK}, to $1.73$ by Davey et al.~\cite{DHDKV}, and to $1.6388$ by Wang~\cite{Wa1}. For other recent progress on strong edge colorings, see~\cite{CHYZ,LL,NN,NY,Wa2,Wa3,WL}.

In an edge-colored graph $G$, a subgraph $H\subseteq G$ is \emph{rainbow} if all the edges in $H$ have distinct colors. In~\cite{GS}, Gy\'arf\'as and S\'ark\"ozy defined the following edge colorings.

\begin{definition}[Gy\'arf\'as and S\'ark\"ozy~\cite{GS}]
For a graph $G$,
\begin{itemize}
    \item an \emph{A-coloring} of $G$ is a proper edge coloring such that the union of any two color classes does not contain any $P_5$ or $C_4$;
    \item a \emph{B-coloring} of $G$ is a proper edge coloring such that every $C_4$ subgraph is rainbow;
    \item a \emph{C-coloring} of $G$ satisfies both conditions for A-coloring and B-coloring.
\end{itemize}
\end{definition}

Each of these definitions is more restrictive than proper edge coloring and less restrictive than strong edge coloring. In fact, for any graph $G$, a strong edge coloring of $G$ is always an A/B/C-coloring of $G$.

A-colorings are also known as \emph{star edge colorings}, and before being defined in~\cite{GS}, they had already been studied from different motivations (see, e.g.,~\cite{BLMSS,DMS,EG}). But the notion of B-coloring was first introduced in~\cite{GS}, and it has since been studied on planar graphs in~\cite{GMRS,KWZ} and on bipartite graphs in~\cite{GSW}. 

Note that, equivalently, we can define a B-coloring of $G=(V,\,E)$ as a proper edge coloring of $G$ such that for any two edges $uv$ and $wx$ with $d(uv,\,wx)=2$, if $\{uw,\,vx\}\subseteq E$ or $\{ux,\,vw\}\subseteq E$, then $uv$ and $wx$ get different colors. So, under a B-coloring, only some specific pairs of edges of distance $2$ have to get different colors --- this is another intuitive way to understand why B-colorings are less restrictive than strong edge colorings.

Motivated by how B-colorings are defined, we introduce a new edge coloring called \emph{D-coloring}. We call it D-coloring because, first, it is defined after A/B/C-colorings, and second, D is the first letter of "diamond".

\begin{definition}
    For a graph $G$, a D-coloring of $G$ is a proper edge coloring such that every diamond subgraph \tikz[baseline=-0.5ex]{
  \node[circle,fill,inner sep=1pt] (a) at (0,0) {};
  \node[circle,fill,inner sep=1pt] (b) at (0.3,0.18) {};
  \node[circle,fill,inner sep=1pt] (c) at (0.3,-0.18) {};
  \node[circle,fill,inner sep=1pt] (d) at (0.6,0) {};
  \draw (a)--(b)--(d)--(c)--(a);
  \draw (b)--(c);
} is rainbow.
\end{definition}

Equivalently, we can also define a D-coloring of $G=(V,\,E)$ as a proper edge coloring such that for any two edges $uv$ and $wx$ with $d(uv,\,wx)=2$, if $|\{uw,\,ux,\,vw,\,vx\}\cap E|\ge 3$, then $uv$ and $wx$ get different colors. 

It is clear that the D-coloring is less restrictive than the B-coloring, because for any graph $G$, a B-coloring of $G$ is always a D-coloring of $G$.

The \emph{B-chromatic index} of $G$, denoted by $\chi'_B(G)$, is the minimum number of colors needed for a B-coloring of $G$. For graphs with restricted maximum degree, as mentioned in~\cite{GS}, by a greedy coloring, we have $\chi'_B(G)\le \Delta^2$ for every graph $G$ with maximum degree $\Delta$. This upper bound is sharp, as it can be attained by the complete bipartite graph $K_{\Delta,\,\Delta}$.

However, for D-colorings, this problem is much more interesting. The \emph{D-chromatic index} of $G$, denoted by $\chi'_D(G)$, is the minimum number of colors needed for a D-coloring of $G$. We propose the following analogue of Conjecture~\ref{conj1}.

\begin{conjecture}\label{conj2}
For every graph $G$ with maximum degree $\Delta$,
\begin{align*}
    \chi'_D(G)\le \frac{1}{2}\Delta^2+\frac{1}{2}\Delta.
\end{align*}
\end{conjecture}

If true, this upper bound will be sharp, because under a D-coloring of the complete graph $K_{\Delta+1}$, every edge has to get a distinct color, which means we need ${\Delta+1 \choose 2}=\frac{1}{2}\Delta^2+\frac{1}{2}\Delta$ colors.

We will verify this conjecture for $\Delta\le 5$.

\begin{theorem}\label{delta5}
    For every graph $G$ with maximum degree $\Delta\le 5$,
\begin{align*}
    \chi'_D(G)\le \frac{1}{2}\Delta^2+\frac{1}{2}\Delta.
\end{align*}
\end{theorem}

Although we cannot completely resolve this conjecture, we can prove an upper bound for $\chi'_D(G)$ with quadratic coefficient $\frac{9}{16}$.

\begin{theorem}\label{general}
    For every graph $G$ with maximum degree $\Delta$,
\begin{align*}
    \chi'_D(G)\le \Bigl\lfloor\frac{9}{16}\Delta^2+\frac{1}{2}\Delta\Bigr\rfloor.
\end{align*}
\end{theorem}

Henceforth, we assume every graph to be connected, because for a disconnected graph $H$ with $k$ components $H_1,\,H_2,\,\dots,\,H_k$, we have $\chi'_D(H)=\max_{1\le i\le k}\chi'_D(H_i)$.

We organize the rest of this paper as follows: In Section~2, we prove the general upper bound in Theorem~\ref{general}. In Section~3, we prove Theorem~\ref{delta5}, which verifies Conjecture~\ref{conj2} for $\Delta\le 5$. In Section~4, we discuss the D-chromatic indices of planar graphs, and propose a related conjecture (Conjecture~\ref{conj3}).

Since the initial version of this paper became publicly available, further developments in D-coloring have emerged. For planar graphs: Hu et al.~\cite{HKW} confirmed Conjecture~\ref{conj3} for $\Delta=4$, $\Delta=5$, and $\Delta\ge 33$. Verma~\cite{Ve} completely settled Conjecture~\ref{conj3} by proving $\chi'_D(G)\le 2\Delta-1$ for planar graphs with $\Delta\ge 6$. For general graphs: Using a probabilistic method, Tian and Wang~\cite{TW} proved that for every graph $G$ with sufficiently large $\Delta$, $\chi_D'(G) \le (1-0.00016)\frac{9}{16}\Delta^2$, improving the bound $\frac{9}{16}\Delta^2+\frac{1}{2}\Delta$ in Theorem~\ref{general}. But to the best of our knowledge, Conjecture~\ref{conj2} remains open for $\Delta\ge 6$.

\section{Proof of Theorem~\ref{general}}

%It is clear that if $H$ is a subgraph of $G$, then $\chi'_D(H)\le \chi'_D(G)$. Combining this fact with the following lemma, we know that, to prove Theorem~\ref{general}, it suffices to prove it for regular graphs.

%\begin{lemma}
%    If the maximum degree of $G$ is $\Delta$, then $G$ is a subgraph of a $\Delta$-regular graph.
%\end{lemma}

%\begin{proof}
%    Let $G=(V,\,E)$ be a graph with maximum degree $\Delta$ and minimum degree $\delta<\Delta$. Let $U\subseteq V$ be the set of vertices with degree $\delta$. We make a copy of $G$ and denote it by $G'$. We denote the copy of $U$ by $U'$, and for each vertex $u\in U$, we denote its copy by $u'$. Now, we construct a new graph $G_1$ by keeping all the vertices and edges in $G$ and $G'$, and connecting each $u\in U$ with its copy $u'\in U'$. Then, $G$ is a subgraph of $G_1$, and the minimum degree of $G_1$ is $\delta+1$. Iterating this process $\Delta-\delta$ times, we get a $\Delta$-regular graph containing $G$ as a subgraph.
%\end{proof}

For a graph $G$ and two edges $e$ and $e'$ in $G$, we say that \emph{$e$ sees $e'$ by $G$} (and $e'$ sees $e$ by $G$) if
\begin{itemize}
    \item $e$ and $e'$ are incident in $G$, \\or
    \item $e$ and $e'$ are not incident, but they are in the same diamond subgraph of $G$.
\end{itemize}
So, if $e$ and $e'$ see each other by $G$, then they must get different colors under a D-coloring of $G$. When our argument involves both $G$ and a subgraph $H\subseteq G$, we will specify whether "$e$ sees $e'$ by $G$" or "$e$ sees $e'$ by $H$". But when the context is clear, we may simply say "$e$ sees $e'$" instead of "$e$ sees $e'$ by $G$".

\begin{proof}[Proof of Theorem~\ref{general}]
    %As mentioned, it suffices to prove Theorem~\ref{general} for regular graphs. 
    The bound is immediate for $\Delta=1$. Let $G$ be a graph with $\Delta\ge 2$ and let $e=xy$ be an edge in $G$. We prove that $e$ sees at most $\bigl\lfloor\frac{9}{16}\Delta^2+\frac{1}{2}\Delta\bigr\rfloor-1$ edges.

    First, as $\deg(x),\,\deg(y)\le \Delta$, we know that $e$ is incident to at most $2\Delta-2$ edges.

    Then, we count the edges disjoint from $e$ that lie in a common diamond subgraph with $e$. For a vertex $v$, denote by $N(v)$ the open neighborhood of $v$. Let $X:=N(x)\setminus \{y\}$, $Y:=N(y)\setminus \{x\}$, $Z:=X\cap Y$, and $\alpha:=|Z|$. If an edge $e'$ is not incident to $e$ but is in a diamond subgraph with $e$, then $e'$ must have 
    \begin{itemize}
        \item both endpoints in $Z$, \\or
        \item one endpoint in $Z$ and the other endpoint in $(X\cup Y)\setminus Z$.
    \end{itemize}
    
    Each vertex in $Z$ is already adjacent to $x$ and $y$, so it has at most $\Delta-2$ more neighbors. Assume that there are $m$ edges with both endpoints in $Z$, and $n$ edges with one endpoint in $Z$ and the other endpoint in $(X\cup Y)\setminus Z$. Our goal is to maximize $m+n$. Each edge with both endpoints in $Z$ contributes $2$ to the degree sum of the vertices in $Z$, and each edge with one endpoint in $Z$ and the other endpoint in $(X\cup Y)\setminus Z$ contributes $1$ to the degree sum of the vertices in $Z$, so we have $2m+n\le \alpha(\Delta-2)$, which implies $m\le \bigl\lfloor\frac{\alpha(\Delta-2)-n}{2}\bigr\rfloor$. Hence $m+n\le \bigl\lfloor\frac{\alpha(\Delta-2)+n}{2}\bigr\rfloor$, and the right-hand side is nondecreasing in $n$.

    We have $|(X\cup Y)\setminus Z|=|X|+|Y|-2\alpha\le 2(\Delta-1-\alpha)$. Depending on the value of $\alpha$, there are two cases.

    \textbf{Case 1.} $\alpha\le \frac{1}{2}\Delta$.

    Since $n\le\alpha(\Delta-2)$, we obtain
    \begin{align}\label{eq1}
        m+n\le \alpha(\Delta-2)\le \frac{1}{2}\Delta(\Delta-2)=\frac{1}{2}\Delta^2-\Delta.
    \end{align}
    
    \textbf{Case 2.} $\alpha>\frac{1}{2}\Delta$.
    
    In this case, we have $\Delta-2>2(\Delta-1-\alpha)$, so each vertex in $Z$ has at most $2(\Delta-1-\alpha)$ neighbors in $(X\cup Y)\setminus Z$. Thus, $n\le2\alpha(\Delta-1-\alpha)$, and we have
    \begin{align*}
        m+n&\le \Bigl\lfloor\alpha^2-\frac{\Delta\alpha}{2}\Bigr\rfloor+\alpha\cdot 2(\Delta-1-\alpha) \\
        &\le \alpha^2-\frac{\Delta\alpha}{2}+\alpha\cdot 2(\Delta-1-\alpha) \\
        &=-\alpha^2+\biggl(\frac{3}{2}\Delta-2\biggr)\alpha.
    \end{align*}
    Then, considering $-\alpha^2+\bigl(\frac{3}{2}\Delta-2\bigr)\alpha$ as a quadratic function of a real variable $\alpha\in\mathbb{R}$, we obtain
    \begin{align}\label{eq2}
        m+n\le -\alpha^2+\biggl(\frac{3}{2}\Delta-2\biggr)\alpha\le \frac{9}{16}\Delta^2-\frac{3}{2}\Delta+1.
    \end{align}

    Comparing \eqref{eq1} and \eqref{eq2}, we always have $\frac{1}{2}\Delta^2-\Delta\le \frac{9}{16}\Delta^2-\frac{3}{2}\Delta+1$, so $m+n\le \frac{9}{16}\Delta^2-\frac{3}{2}\Delta+1$. Thus, $e$ sees at most $\bigl\lfloor(2\Delta-2)+(\frac{9}{16}\Delta^2-\frac{3}{2}\Delta+1)\bigr\rfloor=\bigl\lfloor\frac{9}{16}\Delta^2+\frac{1}{2}\Delta\bigr\rfloor-1$ edges. With $\bigl\lfloor\frac{9}{16}\Delta^2+\frac{1}{2}\Delta\bigr\rfloor$ available colors, we can always find a feasible color for $e$, so $\chi'_D(G)\le \bigl\lfloor\frac{9}{16}\Delta^2+\frac{1}{2}\Delta\bigr\rfloor$.
\end{proof}

In the proof of Theorem~\ref{delta5}, we will also count the number of edges with both endpoints in $Z$, or one endpoint in $Z$ and the other endpoint in $(X\cup Y)\setminus Z$.

\section{Proof of Theorem~\ref{delta5}}

In this section, we prove Theorem~\ref{delta5}. 

If $\Delta=1$ or $\Delta=2$, then $G$ cannot have a diamond subgraph, and we trivially have $\chi'_D(G)=1$ for $\Delta=1$ and $\chi'_D(G)\le 3$ for $\Delta=2$, as desired.

If $\Delta=3$, then, by Theorem~\ref{general}, we have $\chi'_D(G)\le \bigl\lfloor\frac{105}{16}\bigr\rfloor=6$, as desired.

Now, we handle the case $\Delta=4$ or $\Delta=5$. For a graph $G$ and a vertex $v$ in $G$, denote by $G-v$ the subgraph obtained by deleting $v$ and all edges incident to $v$ from $G$.

\begin{lemma}\label{good}
    Let $G$ be a graph and let $v$ be a vertex in $G$. For two edges in $G-v$, if they do not see each other by $G-v$, then they do not see each other by $G$.
\end{lemma}

\begin{proof}
    Assume to the contrary that there exist two edges $e=xy$ and $e'=zw$ in $G-v$ that do not see each other by $G-v$ but see each other by $G$. If $e$ and $e'$ are incident in $G$, then they are also incident in $G-v$, a contradiction. Assume that $e$ and $e'$ are not incident, but they are in the same diamond subgraph of $G$. Then we have $|\{xz,\,xw,\,yz,\,yw\}\cap E(G)|\ge 3$ and $|\{xz,\,xw,\,yz,\,yw\}\cap E(G-v)|<3$, which implies $v\in\{x,\,y,\,z,\,w\}$, contradicting the assumption that both $e$ and $e'$ are in $G-v$.
\end{proof}

For $m,\,n\in\mathbb{N}$ with $m<n$, we use $\intss{m}{n}:=\{x\in\mathbb{N}:m\le x\le n\}$ and $\ints{n}:=\{x\in\mathbb{N}:1\le x\le n\}$ to denote discrete intervals.

If, by $f:E(G)\longrightarrow \ints{k}$, we can use $k$ colors to give $G$ a D-coloring, then $f$ is called a \emph{$k$-D-coloring} of $G$, and $G$ is said to be \emph{$k$-D-colorable}.

If, for a vertex $v\in V(G)$, $f:E(G-v)\longrightarrow \ints{k}$ is a $k$-D-coloring of $G-v$, then $f$ is called a \emph{partial coloring} of $G$. For an edge $e$ incident to $v$, let $L_f(e)$ denote the subset of $\ints{k}$ consisting of the colors that can be used on $e$, i.e.~the colors that are not used on an edge in $G-v$ that sees $e$ by $G$. Let $\ell_f(e):=|L_f(e)|$. It is clear that if $e$ sees $t$ edges in $G-v$ by $G$, then $\ell_f(e)\ge k-t$.

Using the following lemma, we can extend some partial colorings of $G$ to $k$-D-colorings of $G$.

\begin{lemma}\label{extension}
    Let $G$ be a graph and let $v$ be a vertex in $G$ with $n$ edges $e_1,\,e_2,\,\dots,\,e_n$ incident to it. Assume that for $k\in\mathbb{N}$, there is a $k$-D-coloring $f$ of $G-v$. If, up to re-ordering $e_1,\,e_2,\,\dots,\,e_n$, we have
    \begin{align*}
        \ell_f(e_i)\ge i
    \end{align*}
    for each $i\in \ints{n}$, then $f$ can be extended to a $k$-D-coloring of $G$.
\end{lemma}

\begin{proof}
    First, by Lemma~\ref{good}, we can let the edges in $G-v$ keep their colors under $f$. Then, we take $n$ steps to color $e_1,\,e_2,\,\dots,\,e_n$ in order. In the first step, as $\ell_f(e_1)\ge 1$, we can use a color in $L_f(e_1)$ on $e_1$. In the $i$-th step with $i\in \intss{2}{n}$, after coloring $e_1,\,e_2,\,\dots,\,e_{i-1}$, we still have at least $\ell_f(e_i)-(i-1)\ge 1$ color that can be used on $e_i$. Thus, we can color $e_1,\,e_2,\,\dots,\,e_n$ in order and extend $f$ to a $k$-D-coloring of $G$.
\end{proof}

\begin{proof}[Proof of the case $\Delta=4$]
    Assume to the contrary that not every graph with $\Delta=4$ is $10$-D-colorable. Then, there is a \emph{minimal counterexample}, which is a graph $G=(V,E)$ with $\Delta=4$, $\chi'_D(G)\ge 11$, and $\chi'_D(G-a)\le 10$ for every $a\in V$.
    
    If $G$ is not $4$-regular, then there is a vertex $u\in V$ with at most three neighbors. By the minimality of $G$, we know that $G-u$ is $10$-D-colorable. Let $f:E(G-u)\longrightarrow \ints{10}$ be a $10$-D-coloring of $G-u$, so $f$ is a partial coloring of $G$. Let $uv_1,\,\dots,\,uv_{\deg(u)}$ be the $\deg(u)\le 3$ edges incident to $u$. It is shown in the proof of Theorem~\ref{general} that for each $i\in \ints{\deg(u)}$, the number of edges seeing $uv_i$ but not incident to $uv_i$ is at most $\lfloor\frac{9}{16}\Delta^2-\frac{3}{2}\Delta+1\rfloor=4$. Also, there are at most three other edges incident to $uv_i$ at $v_i$. So for each $i\in \ints{\deg(u)}$, $uv_i$ sees at most seven edges in $G-u$ by $G$, which means $\ell_f(uv_i)\ge 10-7=3\ge \deg(u)$. By Lemma~\ref{extension}, we can extend $f$ to a $10$-D-coloring of $G$, contradicting the assumption that $G$ is a counterexample. Thus, we may assume that $G$ is $4$-regular.

    Arbitrarily choose a vertex $u\in V$, let $uv,\ uw_1,\ uw_2,\ uw_3$ be the four edges incident to $u$, and let $vu,\ vx_1,\ vx_2,\ vx_3$ be the four edges incident to $v$. (It is possible that $w_i=x_j$ for any $i,\,j\in \ints{3}$.) By the minimality of $G$, we know that $G-u$ is $10$-D-colorable. Let $f:E(G-u)\longrightarrow \ints{10}$ be a $10$-D-coloring of $G-u$, so $f$ is a partial coloring of $G$.

    First, we show that $\ell_f(uv),\,\ell_f(uw_1),\,\ell_f(uw_2),\,\ell_f(uw_3)\ge 3$. Let $e=uv$, $X=\{w_1,\,w_2,\,w_3\}$, $Y=\{x_1,\,x_2,\,x_3\}$, and $Z=\{w_1,\,w_2,\,w_3\}\cap \{x_1,\,x_2,\,x_3\}$. Then, by the proof of Theorem~\ref{general}, $uv$ sees at most $\bigl\lfloor\frac{9}{16}\Delta^2+\frac{1}{2}\Delta\bigr\rfloor-1=\bigl\lfloor\frac{9}{16}\cdot 4^2+\frac{1}{2}\cdot 4\bigr\rfloor-1=10$ edges by $G$, and as three of them are $uw_1,\,uw_2,\,uw_3$, we know that $uv$ sees at most seven edges in $G-u$ by $G$, which means $\ell_f(uv)\ge 3$. By the same argument, we also have $\ell_f(uw_1),\,\ell_f(uw_2),\,\ell_f(uw_3)\ge 3$.

    If one of $\ell_f(uv),\ \ell_f(uw_1),\ \ell_f(uw_2),\ \ell_f(uw_3)$ is at least $4$, then, by Lemma~\ref{extension}, we can extend $f$ to a $10$-D-coloring of $G$, contradicting the assumption that $G$ is a counterexample. Therefore, we may assume $\ell_f(uv)=\ell_f(uw_1)=\ell_f(uw_2)=\ell_f(uw_3)=3$.

    In order to have $\ell_f(uv)=3$, $uv$ must see seven edges in $G-u$ by $G$. We know that $uv$ sees $vx_1,\,vx_2,\,vx_3$ by $G$, as $uv$ is adjacent to each of them. So, there must be four edges, each of which is not adjacent to $uv$ but is in a diamond subgraph with $uv$. As we have mentioned in the proof of Theorem~\ref{general}, each of these four edges either has both endpoints in $Z$, or has one endpoint in $Z$ and the other endpoint in $(X\cup Y)\setminus Z$.

    If $|Z|=0$, then $uv$ is not in any diamond subgraphs, a contradiction. If $|Z|=1$, then there are at most two edges not incident to $uv$ but each in a diamond subgraph with $uv$, a contradiction. If $|Z|=3$, then there are at most three edges not incident to $uv$ but each in a diamond subgraph with $uv$, a contradiction. Thus, we must have $|Z|=2$. Without loss of generality, we may assume $w_1=x_1$, $w_2=x_2$, but $w_3\neq x_3$. Now, $Z=\{w_1,\,w_2\}$ and $(X\cup Y)\setminus Z=\{w_3,\,x_3\}$. As $G$ is a $4$-regular graph, and both $w_1$ and $w_2$ are already adjacent to $u$ and $v$, we know that each of $w_1$ and $w_2$ has two more neighbors. Hence, in order to have four edges not incident to $uv$ but each in a diamond subgraph with $uv$, we must have $\{w_1w_3,\,w_1x_3,\,w_2w_3,\,w_2x_3\}\subseteq E(G)$. So far, we have determined a subgraph of $G$ shown in Figure~\ref{status1}.
    
    \begin{figure}[H]
        \tikzset{every picture/.style={line width=0.75pt}} %set default line width to 0.75pt        

\begin{tikzpicture}[x=0.75pt,y=0.75pt,yscale=-1,xscale=1]
%uncomment if require: \path (0,1186); %set diagram left start at 0, and has height of 1186

%Shape: Circle [id:dp43230594670872846] 
\draw  [fill={rgb, 255:red, 0; green, 0; blue, 0 }  ,fill opacity=1 ] (135.5,350) .. controls (135.5,346.96) and (137.96,344.5) .. (141,344.5) .. controls (144.04,344.5) and (146.5,346.96) .. (146.5,350) .. controls (146.5,353.04) and (144.04,355.5) .. (141,355.5) .. controls (137.96,355.5) and (135.5,353.04) .. (135.5,350) -- cycle ;
%Shape: Circle [id:dp23275381899545244] 
\draw  [fill={rgb, 255:red, 0; green, 0; blue, 0 }  ,fill opacity=1 ] (205.5,350) .. controls (205.5,346.96) and (207.96,344.5) .. (211,344.5) .. controls (214.04,344.5) and (216.5,346.96) .. (216.5,350) .. controls (216.5,353.04) and (214.04,355.5) .. (211,355.5) .. controls (207.96,355.5) and (205.5,353.04) .. (205.5,350) -- cycle ;
%Shape: Circle [id:dp6190393159151045] 
\draw  [fill={rgb, 255:red, 0; green, 0; blue, 0 }  ,fill opacity=1 ] (65.5,280) .. controls (65.5,276.96) and (67.96,274.5) .. (71,274.5) .. controls (74.04,274.5) and (76.5,276.96) .. (76.5,280) .. controls (76.5,283.04) and (74.04,285.5) .. (71,285.5) .. controls (67.96,285.5) and (65.5,283.04) .. (65.5,280) -- cycle ;
%Shape: Circle [id:dp49866664455824483] 
\draw  [fill={rgb, 255:red, 0; green, 0; blue, 0 }  ,fill opacity=1 ] (135.5,280) .. controls (135.5,276.96) and (137.96,274.5) .. (141,274.5) .. controls (144.04,274.5) and (146.5,276.96) .. (146.5,280) .. controls (146.5,283.04) and (144.04,285.5) .. (141,285.5) .. controls (137.96,285.5) and (135.5,283.04) .. (135.5,280) -- cycle ;
%Shape: Circle [id:dp7257781214166436] 
\draw  [fill={rgb, 255:red, 0; green, 0; blue, 0 }  ,fill opacity=1 ] (205.5,280) .. controls (205.5,276.96) and (207.96,274.5) .. (211,274.5) .. controls (214.04,274.5) and (216.5,276.96) .. (216.5,280) .. controls (216.5,283.04) and (214.04,285.5) .. (211,285.5) .. controls (207.96,285.5) and (205.5,283.04) .. (205.5,280) -- cycle ;
%Shape: Circle [id:dp07807108496555726] 
\draw  [fill={rgb, 255:red, 0; green, 0; blue, 0 }  ,fill opacity=1 ] (275.5,280) .. controls (275.5,276.96) and (277.96,274.5) .. (281,274.5) .. controls (284.04,274.5) and (286.5,276.96) .. (286.5,280) .. controls (286.5,283.04) and (284.04,285.5) .. (281,285.5) .. controls (277.96,285.5) and (275.5,283.04) .. (275.5,280) -- cycle ;
%Straight Lines [id:da41100991559444733] 
\draw    (71,280) -- (141,350) ;
%Straight Lines [id:da17598090701770375] 
\draw    (141,350) -- (211,350) ;
%Straight Lines [id:da7917093052094956] 
\draw [line width=2.25]    (211,350) -- (281,280) ;
%Straight Lines [id:da9020291904661302] 
\draw    (141,350) -- (141,280) ;
%Straight Lines [id:da7121901110590775] 
\draw [line width=2.25]    (211,350) -- (141,280) ;
%Straight Lines [id:da07202115703519885] 
\draw    (141,350) -- (211,280) ;
%Straight Lines [id:da8284715898476113] 
\draw [line width=2.25]    (211,350) -- (211,280) ;
%Curve Lines [id:da5671731011097452] 
\draw [line width=2.25]    (71,280) .. controls (92.2,264.6) and (117.2,263.6) .. (141,280) ;
%Curve Lines [id:da08137641368894499] 
\draw [line width=2.25]    (211,280) .. controls (232.2,264.6) and (257.2,263.6) .. (281,280) ;
%Curve Lines [id:da4022357933467331] 
\draw [line width=2.25]    (71,280) .. controls (105.2,235.6) and (180.2,234.6) .. (211,280) ;
%Curve Lines [id:da1541506008320208] 
\draw [line width=2.25]    (141,280) .. controls (175.2,235.6) and (250.2,234.6) .. (281,280) ;

% Text Node
\draw (134,359) node [anchor=north west][inner sep=0.75pt]  [color={rgb, 255:red, 0; green, 0; blue, 0 }  ,opacity=1 ] [align=left] {$\displaystyle u$};
% Text Node
\draw (206,359) node [anchor=north west][inner sep=0.75pt]  [color={rgb, 255:red, 0; green, 0; blue, 0 }  ,opacity=1 ] [align=left] {$\displaystyle v$};
% Text Node
\draw (44,276) node [anchor=north west][inner sep=0.75pt]  [color={rgb, 255:red, 0; green, 0; blue, 0 }  ,opacity=1 ] [align=left] {$\displaystyle w_{3}$};
% Text Node
\draw (149,276) node [anchor=north west][inner sep=0.75pt]  [color={rgb, 255:red, 0; green, 0; blue, 0 }  ,opacity=1 ] [align=left] {$\displaystyle w_{1}$};
% Text Node
\draw (219,276) node [anchor=north west][inner sep=0.75pt]  [color={rgb, 255:red, 0; green, 0; blue, 0 }  ,opacity=1 ] [align=left] {$\displaystyle w_{2}$};
% Text Node
\draw (290,276) node [anchor=north west][inner sep=0.75pt]  [color={rgb, 255:red, 0; green, 0; blue, 0 }  ,opacity=1 ] [align=left] {$\displaystyle x_{3}$};

\end{tikzpicture}
\caption{$uv$ sees the seven bold edges in $G-u$ by $G$.}
\label{status1}
    \end{figure}
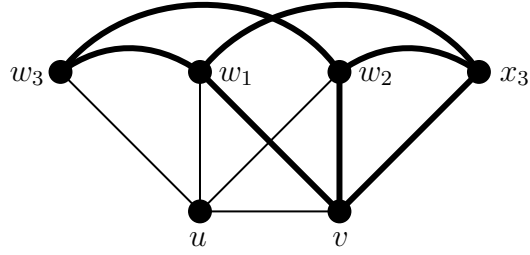

    To show the structure of the edges incident to $uw_1$ more clearly, Figure~\ref{status1} is redrawn as Figure~\ref{status2}. As shown in Figure~\ref{status2}, $uw_1$ already sees six edges in $G-u$ by $G$.

    \begin{figure}[H]
        \tikzset{every picture/.style={line width=0.75pt}} %set default line width to 0.75pt        

\begin{tikzpicture}[x=0.75pt,y=0.75pt,yscale=-1,xscale=1]
%uncomment if require: \path (0,1186); %set diagram left start at 0, and has height of 1186

%Shape: Circle [id:dp6968336864680434] 
\draw  [fill={rgb, 255:red, 0; green, 0; blue, 0 }  ,fill opacity=1 ] (466.5,419) .. controls (466.5,415.96) and (468.96,413.5) .. (472,413.5) .. controls (475.04,413.5) and (477.5,415.96) .. (477.5,419) .. controls (477.5,422.04) and (475.04,424.5) .. (472,424.5) .. controls (468.96,424.5) and (466.5,422.04) .. (466.5,419) -- cycle ;
%Shape: Circle [id:dp0016659604506253922] 
\draw  [fill={rgb, 255:red, 0; green, 0; blue, 0 }  ,fill opacity=1 ] (536.5,419) .. controls (536.5,415.96) and (538.96,413.5) .. (542,413.5) .. controls (545.04,413.5) and (547.5,415.96) .. (547.5,419) .. controls (547.5,422.04) and (545.04,424.5) .. (542,424.5) .. controls (538.96,424.5) and (536.5,422.04) .. (536.5,419) -- cycle ;
%Shape: Circle [id:dp6068091033049852] 
\draw  [fill={rgb, 255:red, 0; green, 0; blue, 0 }  ,fill opacity=1 ] (396.5,349) .. controls (396.5,345.96) and (398.96,343.5) .. (402,343.5) .. controls (405.04,343.5) and (407.5,345.96) .. (407.5,349) .. controls (407.5,352.04) and (405.04,354.5) .. (402,354.5) .. controls (398.96,354.5) and (396.5,352.04) .. (396.5,349) -- cycle ;
%Shape: Circle [id:dp779506227286651] 
\draw  [fill={rgb, 255:red, 0; green, 0; blue, 0 }  ,fill opacity=1 ] (466.5,349) .. controls (466.5,345.96) and (468.96,343.5) .. (472,343.5) .. controls (475.04,343.5) and (477.5,345.96) .. (477.5,349) .. controls (477.5,352.04) and (475.04,354.5) .. (472,354.5) .. controls (468.96,354.5) and (466.5,352.04) .. (466.5,349) -- cycle ;
%Shape: Circle [id:dp3564342636482223] 
\draw  [fill={rgb, 255:red, 0; green, 0; blue, 0 }  ,fill opacity=1 ] (536.5,349) .. controls (536.5,345.96) and (538.96,343.5) .. (542,343.5) .. controls (545.04,343.5) and (547.5,345.96) .. (547.5,349) .. controls (547.5,352.04) and (545.04,354.5) .. (542,354.5) .. controls (538.96,354.5) and (536.5,352.04) .. (536.5,349) -- cycle ;
%Shape: Circle [id:dp022645066422550597] 
\draw  [fill={rgb, 255:red, 0; green, 0; blue, 0 }  ,fill opacity=1 ] (606.5,349) .. controls (606.5,345.96) and (608.96,343.5) .. (612,343.5) .. controls (615.04,343.5) and (617.5,345.96) .. (617.5,349) .. controls (617.5,352.04) and (615.04,354.5) .. (612,354.5) .. controls (608.96,354.5) and (606.5,352.04) .. (606.5,349) -- cycle ;
%Straight Lines [id:da13636457703998373] 
\draw    (402,349) -- (472,419) ;
%Straight Lines [id:da6183693067600932] 
\draw    (472,419) -- (542,419) ;
%Straight Lines [id:da27339833538832836] 
\draw [line width=2.25]    (542,419) -- (612,349) ;
%Straight Lines [id:da8821354181787334] 
\draw    (472,419) -- (472,349) ;
%Straight Lines [id:da09987658343189132] 
\draw [line width=2.25]    (542,419) -- (472,349) ;
%Straight Lines [id:da4277020307492111] 
\draw    (472,419) -- (542,349) ;
%Straight Lines [id:da5284189054047002] 
\draw [line width=2.25]    (542,419) -- (542,349) ;
%Curve Lines [id:da968424930124565] 
\draw [line width=2.25]    (402,349) .. controls (423.2,333.6) and (448.2,332.6) .. (472,349) ;
%Curve Lines [id:da24655774347403336] 
\draw [line width=2.25]    (542,349) .. controls (563.2,333.6) and (588.2,332.6) .. (612,349) ;
%Curve Lines [id:da10930462797103813] 
\draw [line width=2.25]    (402,349) .. controls (436.2,304.6) and (511.2,303.6) .. (542,349) ;
%Curve Lines [id:da03962769208953487] 
\draw    (402,349) .. controls (450.1,260.36) and (566.1,259.36) .. (612,349) ;

% Text Node
\draw (465,428) node [anchor=north west][inner sep=0.75pt]  [color={rgb, 255:red, 0; green, 0; blue, 0 }  ,opacity=1 ] [align=left] {$\displaystyle u$};
% Text Node
\draw (537,428) node [anchor=north west][inner sep=0.75pt]  [color={rgb, 255:red, 0; green, 0; blue, 0 }  ,opacity=1 ] [align=left] {$\displaystyle w_{1}$};
% Text Node
\draw (375,345) node [anchor=north west][inner sep=0.75pt]  [color={rgb, 255:red, 0; green, 0; blue, 0 }  ,opacity=1 ] [align=left] {$\displaystyle w_{2}$};
% Text Node
\draw (480,345) node [anchor=north west][inner sep=0.75pt]  [color={rgb, 255:red, 0; green, 0; blue, 0 }  ,opacity=1 ] [align=left] {$\displaystyle w_{3}$};
% Text Node
\draw (550,347) node [anchor=north west][inner sep=0.75pt]  [color={rgb, 255:red, 0; green, 0; blue, 0 }  ,opacity=1 ] [align=left] {$\displaystyle v$};
% Text Node
\draw (621,345) node [anchor=north west][inner sep=0.75pt]  [color={rgb, 255:red, 0; green, 0; blue, 0 }  ,opacity=1 ] [align=left] {$\displaystyle x_{3}$};

\end{tikzpicture}
\caption{$uw_1$ sees the six bold edges in $G-u$ by $G$.}
\label{status2}
    \end{figure}
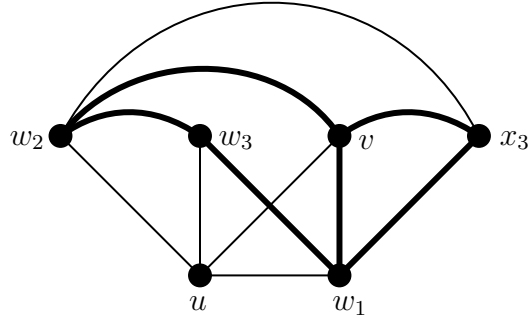

    As $G$ is $4$-regular, each of $w_3$ and $x_3$ needs to have one more neighbor. If $w_3$ and $x_3$ are not adjacent, then $uw_1$ only sees six edges in $G-u$ by $G$, which means $\ell_f(uw_1)\ge 4$, a contradiction. So $w_3$ and $x_3$ are adjacent. But now, each of the six vertices in our graph has degree $4$, and as $G$ is a connected $4$-regular graph, the structure of $G$ has been determined --- as shown in Figure~\ref{status3}(a), $G$ is a graph that can be embedded on a regular octahedron. However, as shown in Figure~\ref{status3}(b), $G$ is $6$-D-colorable, a contradiction.

    \begin{figure}[H]
        \tikzset{every picture/.style={line width=0.75pt}} %set default line width to 0.75pt        

\begin{tikzpicture}[x=0.75pt,y=0.75pt,yscale=-1,xscale=1]
%uncomment if require: \path (0,1535); %set diagram left start at 0, and has height of 1535

%Shape: Parallelogram [id:dp3999478514442797] 
\draw   (183.48,561.2) -- (285.6,561.2) -- (216.12,600.6) -- (114,600.6) -- cycle ;
%Shape: Circle [id:dp6403671949312965] 
\draw  [fill={rgb, 255:red, 0; green, 0; blue, 0 }  ,fill opacity=1 ] (194.3,497.2) .. controls (194.3,494.16) and (196.76,491.7) .. (199.8,491.7) .. controls (202.84,491.7) and (205.3,494.16) .. (205.3,497.2) .. controls (205.3,500.24) and (202.84,502.7) .. (199.8,502.7) .. controls (196.76,502.7) and (194.3,500.24) .. (194.3,497.2) -- cycle ;
%Shape: Circle [id:dp3058615986070633] 
\draw  [fill={rgb, 255:red, 0; green, 0; blue, 0 }  ,fill opacity=1 ] (177.98,561.2) .. controls (177.98,558.16) and (180.45,555.7) .. (183.48,555.7) .. controls (186.52,555.7) and (188.98,558.16) .. (188.98,561.2) .. controls (188.98,564.24) and (186.52,566.7) .. (183.48,566.7) .. controls (180.45,566.7) and (177.98,564.24) .. (177.98,561.2) -- cycle ;
%Shape: Circle [id:dp4233534237853186] 
\draw  [fill={rgb, 255:red, 0; green, 0; blue, 0 }  ,fill opacity=1 ] (108.5,600.6) .. controls (108.5,597.56) and (110.96,595.1) .. (114,595.1) .. controls (117.04,595.1) and (119.5,597.56) .. (119.5,600.6) .. controls (119.5,603.64) and (117.04,606.1) .. (114,606.1) .. controls (110.96,606.1) and (108.5,603.64) .. (108.5,600.6) -- cycle ;
%Shape: Circle [id:dp7392711358739618] 
\draw  [fill={rgb, 255:red, 0; green, 0; blue, 0 }  ,fill opacity=1 ] (210.62,600.6) .. controls (210.62,597.56) and (213.08,595.1) .. (216.12,595.1) .. controls (219.15,595.1) and (221.62,597.56) .. (221.62,600.6) .. controls (221.62,603.64) and (219.15,606.1) .. (216.12,606.1) .. controls (213.08,606.1) and (210.62,603.64) .. (210.62,600.6) -- cycle ;
%Shape: Circle [id:dp24645689008544147] 
\draw  [fill={rgb, 255:red, 0; green, 0; blue, 0 }  ,fill opacity=1 ] (280.1,561.2) .. controls (280.1,558.16) and (282.56,555.7) .. (285.6,555.7) .. controls (288.64,555.7) and (291.1,558.16) .. (291.1,561.2) .. controls (291.1,564.24) and (288.64,566.7) .. (285.6,566.7) .. controls (282.56,566.7) and (280.1,564.24) .. (280.1,561.2) -- cycle ;
%Shape: Circle [id:dp886747512716974] 
\draw  [fill={rgb, 255:red, 0; green, 0; blue, 0 }  ,fill opacity=1 ] (194.3,664.6) .. controls (194.3,661.56) and (196.76,659.1) .. (199.8,659.1) .. controls (202.84,659.1) and (205.3,661.56) .. (205.3,664.6) .. controls (205.3,667.64) and (202.84,670.1) .. (199.8,670.1) .. controls (196.76,670.1) and (194.3,667.64) .. (194.3,664.6) -- cycle ;
%Straight Lines [id:da7254039486666274] 
\draw    (199.8,497.2) -- (114,600.6) ;
%Straight Lines [id:da9432493108806523] 
\draw    (199.8,497.2) -- (183.48,561.2) ;
%Straight Lines [id:da5292064621331938] 
\draw    (199.8,497.2) -- (216.12,600.6) ;
%Straight Lines [id:da8054353531243066] 
\draw    (199.8,497.2) -- (285.6,561.2) ;
%Straight Lines [id:da042401438635126576] 
\draw    (114,600.6) -- (199.8,664.6) ;
%Straight Lines [id:da11472636220887344] 
\draw    (183.48,561.2) -- (199.8,664.6) ;
%Straight Lines [id:da6009485653926653] 
\draw    (216.12,600.6) -- (199.8,664.6) ;
%Straight Lines [id:da20225079343307695] 
\draw    (285.6,561.2) -- (199.8,664.6) ;
%Shape: Circle [id:dp2540626452788326] 
\draw  [fill={rgb, 255:red, 0; green, 0; blue, 0 }  ,fill opacity=1 ] (437.3,497.2) .. controls (437.3,494.16) and (439.76,491.7) .. (442.8,491.7) .. controls (445.84,491.7) and (448.3,494.16) .. (448.3,497.2) .. controls (448.3,500.24) and (445.84,502.7) .. (442.8,502.7) .. controls (439.76,502.7) and (437.3,500.24) .. (437.3,497.2) -- cycle ;
%Shape: Circle [id:dp6036325526573264] 
\draw  [fill={rgb, 255:red, 0; green, 0; blue, 0 }  ,fill opacity=1 ] (420.98,561.2) .. controls (420.98,558.16) and (423.45,555.7) .. (426.48,555.7) .. controls (429.52,555.7) and (431.98,558.16) .. (431.98,561.2) .. controls (431.98,564.24) and (429.52,566.7) .. (426.48,566.7) .. controls (423.45,566.7) and (420.98,564.24) .. (420.98,561.2) -- cycle ;
%Shape: Circle [id:dp664647319541885] 
\draw  [fill={rgb, 255:red, 0; green, 0; blue, 0 }  ,fill opacity=1 ] (351.5,600.6) .. controls (351.5,597.56) and (353.96,595.1) .. (357,595.1) .. controls (360.04,595.1) and (362.5,597.56) .. (362.5,600.6) .. controls (362.5,603.64) and (360.04,606.1) .. (357,606.1) .. controls (353.96,606.1) and (351.5,603.64) .. (351.5,600.6) -- cycle ;
%Shape: Circle [id:dp25774417134951433] 
\draw  [fill={rgb, 255:red, 0; green, 0; blue, 0 }  ,fill opacity=1 ] (453.62,600.6) .. controls (453.62,597.56) and (456.08,595.1) .. (459.12,595.1) .. controls (462.15,595.1) and (464.62,597.56) .. (464.62,600.6) .. controls (464.62,603.64) and (462.15,606.1) .. (459.12,606.1) .. controls (456.08,606.1) and (453.62,603.64) .. (453.62,600.6) -- cycle ;
%Shape: Circle [id:dp74274550271306] 
\draw  [fill={rgb, 255:red, 0; green, 0; blue, 0 }  ,fill opacity=1 ] (523.1,561.2) .. controls (523.1,558.16) and (525.56,555.7) .. (528.6,555.7) .. controls (531.64,555.7) and (534.1,558.16) .. (534.1,561.2) .. controls (534.1,564.24) and (531.64,566.7) .. (528.6,566.7) .. controls (525.56,566.7) and (523.1,564.24) .. (523.1,561.2) -- cycle ;
%Shape: Circle [id:dp7952511535779759] 
\draw  [fill={rgb, 255:red, 0; green, 0; blue, 0 }  ,fill opacity=1 ] (437.3,664.6) .. controls (437.3,661.56) and (439.76,659.1) .. (442.8,659.1) .. controls (445.84,659.1) and (448.3,661.56) .. (448.3,664.6) .. controls (448.3,667.64) and (445.84,670.1) .. (442.8,670.1) .. controls (439.76,670.1) and (437.3,667.64) .. (437.3,664.6) -- cycle ;
%Straight Lines [id:da9325342325416927] 
\draw [color={rgb, 255:red, 208; green, 2; blue, 27 }  ,draw opacity=1 ]   (442.8,497.2) -- (357,600.6) ;
%Straight Lines [id:da6252158308177805] 
\draw [color={rgb, 255:red, 245; green, 166; blue, 35 }  ,draw opacity=1 ]   (442.8,497.2) -- (426.48,561.2) ;
%Straight Lines [id:da43161866218883316] 
\draw [color={rgb, 255:red, 144; green, 19; blue, 254 }  ,draw opacity=1 ]   (442.8,497.2) -- (459.12,600.6) ;
%Straight Lines [id:da10751316505255182] 
\draw [color={rgb, 255:red, 74; green, 144; blue, 226 }  ,draw opacity=1 ]   (442.8,497.2) -- (528.6,561.2) ;
%Straight Lines [id:da6525381775684835] 
\draw [color={rgb, 255:red, 74; green, 144; blue, 226 }  ,draw opacity=1 ]   (357,600.6) -- (442.8,664.6) ;
%Straight Lines [id:da6189946016249229] 
\draw [color={rgb, 255:red, 144; green, 19; blue, 254 }  ,draw opacity=1 ]   (426.48,561.2) -- (442.8,664.6) ;
%Straight Lines [id:da0013258415004289148] 
\draw [color={rgb, 255:red, 245; green, 166; blue, 35 }  ,draw opacity=1 ]   (459.12,600.6) -- (442.8,664.6) ;
%Straight Lines [id:da6296560472345799] 
\draw [color={rgb, 255:red, 208; green, 2; blue, 27 }  ,draw opacity=1 ]   (528.6,561.2) -- (442.8,664.6) ;
%Straight Lines [id:da9222312418866446] 
\draw [color={rgb, 255:red, 65; green, 117; blue, 5 }  ,draw opacity=1 ]   (357,600.6) -- (459.12,600.6) ;
%Straight Lines [id:da21273261020017764] 
\draw [color={rgb, 255:red, 65; green, 117; blue, 5 }  ,draw opacity=1 ]   (426.48,561.2) -- (528.6,561.2) ;
%Straight Lines [id:da31930944828420094] 
\draw [color={rgb, 255:red, 80; green, 227; blue, 194 }  ,draw opacity=1 ]   (426.48,561.2) -- (357,600.6) ;
%Straight Lines [id:da20341478004391267] 
\draw [color={rgb, 255:red, 80; green, 227; blue, 194 }  ,draw opacity=1 ]   (528.6,561.2) -- (459.12,600.6) ;

% Text Node
\draw (93,597) node [anchor=north west][inner sep=0.75pt]  [color={rgb, 255:red, 0; green, 0; blue, 0 }  ,opacity=1 ] [align=left] {$\displaystyle u$};
% Text Node
\draw (208,492) node [anchor=north west][inner sep=0.75pt]  [color={rgb, 255:red, 0; green, 0; blue, 0 }  ,opacity=1 ] [align=left] {$\displaystyle v$};
% Text Node
\draw (157.9,549.9) node [anchor=north west][inner sep=0.75pt]  [color={rgb, 255:red, 0; green, 0; blue, 0 }  ,opacity=1 ] [align=left] {$\displaystyle w_{1}$};
% Text Node
\draw (221.62,600.6) node [anchor=north west][inner sep=0.75pt]  [color={rgb, 255:red, 0; green, 0; blue, 0 }  ,opacity=1 ] [align=left] {$\displaystyle w_{2}$};
% Text Node
\draw (190.62,673.6) node [anchor=north west][inner sep=0.75pt]  [color={rgb, 255:red, 0; green, 0; blue, 0 }  ,opacity=1 ] [align=left] {$\displaystyle w_{3}$};
% Text Node
\draw (293.62,556.6) node [anchor=north west][inner sep=0.75pt]  [color={rgb, 255:red, 0; green, 0; blue, 0 }  ,opacity=1 ] [align=left] {$\displaystyle x_{3}$};
% Text Node
\draw (336,597) node [anchor=north west][inner sep=0.75pt]  [color={rgb, 255:red, 0; green, 0; blue, 0 }  ,opacity=1 ] [align=left] {$\displaystyle u$};
% Text Node
\draw (451,492) node [anchor=north west][inner sep=0.75pt]  [color={rgb, 255:red, 0; green, 0; blue, 0 }  ,opacity=1 ] [align=left] {$\displaystyle v$};
% Text Node
\draw (400.9,549.9) node [anchor=north west][inner sep=0.75pt]  [color={rgb, 255:red, 0; green, 0; blue, 0 }  ,opacity=1 ] [align=left] {$\displaystyle w_{1}$};
% Text Node
\draw (464.62,600.6) node [anchor=north west][inner sep=0.75pt]  [color={rgb, 255:red, 0; green, 0; blue, 0 }  ,opacity=1 ] [align=left] {$\displaystyle w_{2}$};
% Text Node
\draw (433.62,673.6) node [anchor=north west][inner sep=0.75pt]  [color={rgb, 255:red, 0; green, 0; blue, 0 }  ,opacity=1 ] [align=left] {$\displaystyle w_{3}$};
% Text Node
\draw (536.62,556.6) node [anchor=north west][inner sep=0.75pt]  [color={rgb, 255:red, 0; green, 0; blue, 0 }  ,opacity=1 ] [align=left] {$\displaystyle x_{3}$};
% Text Node
\draw (191,705) node [anchor=north west][inner sep=0.75pt]   [align=left] {(a)};
% Text Node
\draw (434,704) node [anchor=north west][inner sep=0.75pt]   [align=left] {(b)};

\end{tikzpicture}
\caption{(a) The structure of $G$; (b) A $6$-D-coloring of $G$.}
\label{status3}
    \end{figure}
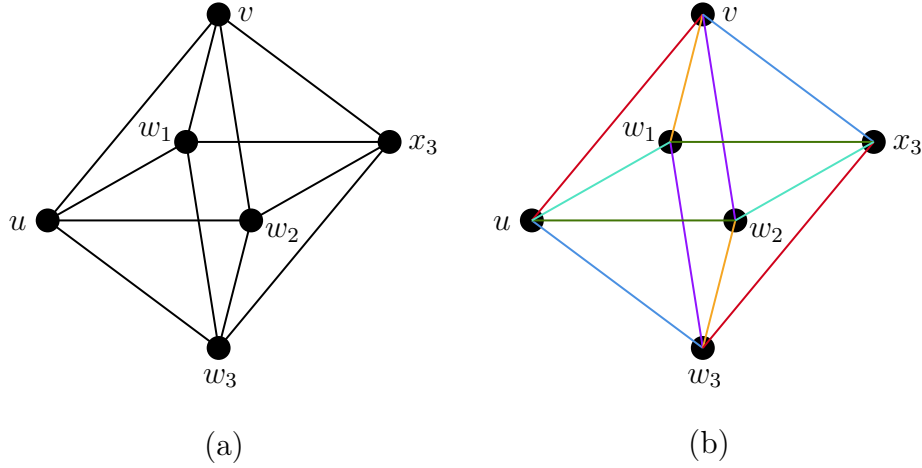

    Therefore, such a minimal counterexample $G$ does not exist, and the case $\Delta=4$ of Theorem~\ref{delta5} is established.
\end{proof}

\begin{proof}[Proof of the case $\Delta=5$]
    Assume to the contrary that not every graph with $\Delta=5$ is $15$-D-colorable. Then, there is a \emph{minimal counterexample}, which is a graph $G=(V,E)$ with $\Delta=5$, $\chi'_D(G)\ge 16$, and $\chi'_D(G-a)\le 15$ for every $a\in V$. 
    
    If $G$ is not $5$-regular, then there is a vertex $u\in V$ with at most four neighbors. By the minimality of $G$, we know that $G-u$ is $15$-D-colorable. Let $f:E(G-u)\longrightarrow \ints{15}$ be a $15$-D-coloring of $G-u$, so $f$ is a partial coloring of $G$. Let $uv_1,\,\dots,\,uv_{\deg(u)}$ be the $\deg(u)\le 4$ edges incident to $u$. It is shown in the proof of Theorem~\ref{general} that for each $i\in \ints{\deg(u)}$, the number of edges seeing $uv_i$ but not incident to $uv_i$ is at most $\lfloor\frac{9}{16}\Delta^2-\frac{3}{2}\Delta+1\rfloor=7$. Also, there are at most four other edges incident to $uv_i$ at $v_i$. So for each $i\in \ints{\deg(u)}$, $uv_i$ sees at most $11$ edges in $G-u$ by $G$, which means $\ell_f(uv_i)\ge 15-11=4\ge \deg(u)$. By Lemma~\ref{extension}, we can extend $f$ to a $15$-D-coloring of $G$, contradicting the assumption that $G$ is a counterexample. Thus, we may assume that $G$ is $5$-regular.

    Arbitrarily choose a vertex $u\in V$, let $uv,\ uw_1,\ uw_2,\ uw_3,\ uw_4$ be the five edges incident to $u$, and let $vu,\ vx_1,\ vx_2,\ vx_3,\ vx_4$ be the five edges incident to $v$. (It is possible that $w_i=x_j$ for any $i,\,j\in \ints{4}$.) By the minimality of $G$, we know that $G-u$ is $15$-D-colorable. Let $f:E(G-u)\longrightarrow \ints{15}$ be a $15$-D-coloring of $G-u$, so $f$ is a partial coloring of $G$.

    First, we show that $\ell_f(uv),\,\ell_f(uw_1),\,\ell_f(uw_2),\,\ell_f(uw_3),\,\ell_f(uw_4)\ge 4$. Let $e=uv$, $X=\{w_1,\,w_2,\,w_3,\,w_4\}$, $Y=\{x_1,\,x_2,\,x_3,\,x_4\}$, and $Z=\{w_1,\,w_2,\,w_3,\,w_4\}\cap \{x_1,\,x_2,\,x_3,\,x_4\}$. Then, by the proof of Theorem~\ref{general}, $uv$ sees at most $\bigl\lfloor\frac{9}{16}\Delta^2+\frac{1}{2}\Delta\bigr\rfloor-1=\bigl\lfloor\frac{9}{16}\cdot 5^2+\frac{1}{2}\cdot 5\bigr\rfloor-1=15$ edges by $G$, and as four of them are $uw_1,\,uw_2,\,uw_3,\,uw_4$, we know that $uv$ sees at most $11$ edges in $G-u$ by $G$, which means $\ell_f(uv)\ge 4$. By the same argument, we also have $\ell_f(uw_1),\,\ell_f(uw_2),\,\ell_f(uw_3),\,\ell_f(uw_4)\ge 4$.

    If one of $\ell_f(uv),\ \ell_f(uw_1),\ \ell_f(uw_2),\ \ell_f(uw_3),\ \ell_f(uw_4)$ is at least $5$, then, by Lemma~\ref{extension}, we can extend $f$ to a $15$-D-coloring of $G$, contradicting the assumption that $G$ is a counterexample. Therefore, we may assume $\ell_f(uv)=\ell_f(uw_1)=\ell_f(uw_2)=\ell_f(uw_3)=\ell_f(uw_4)=4$.

    In order to have $\ell_f(uv)=4$, $uv$ must see $11$ edges in $G-u$ by $G$. We know that $uv$ sees $vx_1,\,vx_2,\,vx_3,\,vx_4$ by $G$, as $uv$ is adjacent to each of them. So, there must be seven edges, each of which is not adjacent to $uv$ but is in a diamond subgraph with $uv$. As we have mentioned in the proof of Theorem~\ref{general}, each of these seven edges either has both endpoints in $Z$, or has one endpoint in $Z$ and the other endpoint in $(X\cup Y)\setminus Z$.

    If $|Z|=0$, then $uv$ is not in any diamond subgraphs, a contradiction. If $|Z|=1$, then there are at most three edges not incident to $uv$ but each in a diamond subgraph with $uv$, a contradiction. If $|Z|=2$, then there are at most six edges not incident to $uv$ but each in a diamond subgraph with $uv$, a contradiction. If $|Z|=4$, then there are at most six edges not incident to $uv$ but each in a diamond subgraph with $uv$, a contradiction. Thus, we must have $|Z|=3$. Without loss of generality, we may assume $w_1=x_1$, $w_2=x_2$, $w_3=x_3$, but $w_4\neq x_4$. Now we have $Z=\{w_1,\,w_2,\,w_3\}$ and $(X\cup Y)\setminus Z=\{w_4,\,x_4\}$. Depending on the number of edges with both endpoints in $Z$, there are four cases.

    \textbf{Case 1.} There are no edges with both endpoints in $Z=\{w_1,\,w_2,\,w_3\}$.

    In this case, we need to have seven edges each with one endpoint in $\{w_1,\,w_2,\,w_3\}$ and the other endpoint in $\{w_4,\,x_4\}$, which is impossible.

    \textbf{Case 2.} There is exactly one edge with both endpoints in $Z=\{w_1,\,w_2,\,w_3\}$.

    By symmetry, we may assume that $w_1 w_2$ is the only edge with both endpoints in $Z$. Then, there are six edges each with one endpoint in $\{w_1,\,w_2,\,w_3\}$ and the other endpoint in $\{w_4,\,x_4\}$, which means $\{w_1 w_4,\,w_1 x_4,\,w_2 w_4,\,w_2 x_4,\,w_3 w_4,\,w_3 x_4\}\subseteq E(G)$. So far, we have determined a subgraph of $G$ shown in Figure~\ref{fig1}.

    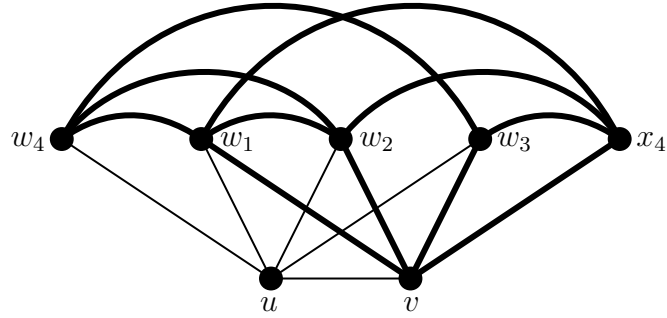
\begin{figure}[H]
        \tikzset{every picture/.style={line width=0.75pt}} %set default line width to 0.75pt        

\begin{tikzpicture}[x=0.75pt,y=0.75pt,yscale=-1,xscale=1]
%uncomment if require: \path (0,1905); %set diagram left start at 0, and has height of 1905

%Shape: Circle [id:dp18378326685173163] 
\draw  [fill={rgb, 255:red, 0; green, 0; blue, 0 }  ,fill opacity=1 ] (192.5,437) .. controls (192.5,433.96) and (194.96,431.5) .. (198,431.5) .. controls (201.04,431.5) and (203.5,433.96) .. (203.5,437) .. controls (203.5,440.04) and (201.04,442.5) .. (198,442.5) .. controls (194.96,442.5) and (192.5,440.04) .. (192.5,437) -- cycle ;
%Shape: Circle [id:dp7291459111202497] 
\draw  [fill={rgb, 255:red, 0; green, 0; blue, 0 }  ,fill opacity=1 ] (262.5,437) .. controls (262.5,433.96) and (264.96,431.5) .. (268,431.5) .. controls (271.04,431.5) and (273.5,433.96) .. (273.5,437) .. controls (273.5,440.04) and (271.04,442.5) .. (268,442.5) .. controls (264.96,442.5) and (262.5,440.04) .. (262.5,437) -- cycle ;
%Shape: Circle [id:dp7764142208972562] 
\draw  [fill={rgb, 255:red, 0; green, 0; blue, 0 }  ,fill opacity=1 ] (87.5,367) .. controls (87.5,363.96) and (89.96,361.5) .. (93,361.5) .. controls (96.04,361.5) and (98.5,363.96) .. (98.5,367) .. controls (98.5,370.04) and (96.04,372.5) .. (93,372.5) .. controls (89.96,372.5) and (87.5,370.04) .. (87.5,367) -- cycle ;
%Shape: Circle [id:dp5562582824504767] 
\draw  [fill={rgb, 255:red, 0; green, 0; blue, 0 }  ,fill opacity=1 ] (157.5,367) .. controls (157.5,363.96) and (159.96,361.5) .. (163,361.5) .. controls (166.04,361.5) and (168.5,363.96) .. (168.5,367) .. controls (168.5,370.04) and (166.04,372.5) .. (163,372.5) .. controls (159.96,372.5) and (157.5,370.04) .. (157.5,367) -- cycle ;
%Shape: Circle [id:dp4611132558642759] 
\draw  [fill={rgb, 255:red, 0; green, 0; blue, 0 }  ,fill opacity=1 ] (227.5,367) .. controls (227.5,363.96) and (229.96,361.5) .. (233,361.5) .. controls (236.04,361.5) and (238.5,363.96) .. (238.5,367) .. controls (238.5,370.04) and (236.04,372.5) .. (233,372.5) .. controls (229.96,372.5) and (227.5,370.04) .. (227.5,367) -- cycle ;
%Shape: Circle [id:dp08720536839475412] 
\draw  [fill={rgb, 255:red, 0; green, 0; blue, 0 }  ,fill opacity=1 ] (297.5,367) .. controls (297.5,363.96) and (299.96,361.5) .. (303,361.5) .. controls (306.04,361.5) and (308.5,363.96) .. (308.5,367) .. controls (308.5,370.04) and (306.04,372.5) .. (303,372.5) .. controls (299.96,372.5) and (297.5,370.04) .. (297.5,367) -- cycle ;
%Shape: Circle [id:dp43886414738406854] 
\draw  [fill={rgb, 255:red, 0; green, 0; blue, 0 }  ,fill opacity=1 ] (367.5,367) .. controls (367.5,363.96) and (369.96,361.5) .. (373,361.5) .. controls (376.04,361.5) and (378.5,363.96) .. (378.5,367) .. controls (378.5,370.04) and (376.04,372.5) .. (373,372.5) .. controls (369.96,372.5) and (367.5,370.04) .. (367.5,367) -- cycle ;
%Straight Lines [id:da40507519559314453] 
\draw    (93,367) -- (198,437) ;
%Straight Lines [id:da15827475828932325] 
\draw    (163,367) -- (198,437) ;
%Straight Lines [id:da3590052353152159] 
\draw    (233,367) -- (198,437) ;
%Straight Lines [id:da28675620901117604] 
\draw    (303,367) -- (198,437) ;
%Straight Lines [id:da3246381820473736] 
\draw [line width=2.25]    (163,367) -- (268,437) ;
%Straight Lines [id:da7059281132498296] 
\draw [line width=2.25]    (233,367) -- (268,437) ;
%Straight Lines [id:da7381438653268706] 
\draw [line width=2.25]    (373,367) -- (268,437) ;
%Straight Lines [id:da49817674283261615] 
\draw [line width=2.25]    (303,367) -- (268,437) ;
%Curve Lines [id:da8992335354574972] 
\draw [line width=2.25]    (93,367) .. controls (114.2,351.6) and (139.2,350.6) .. (163,367) ;
%Curve Lines [id:da7622607909602611] 
\draw [line width=2.25]    (163,367) .. controls (184.2,351.6) and (209.2,350.6) .. (233,367) ;
%Curve Lines [id:da2472485572724742] 
\draw [line width=2.25]    (303,367) .. controls (324.2,351.6) and (349.2,350.6) .. (373,367) ;
%Curve Lines [id:da6930685978082338] 
\draw [line width=2.25]    (93,367) .. controls (127.2,322.6) and (202.2,321.6) .. (233,367) ;
%Curve Lines [id:da7411538814042209] 
\draw [line width=2.25]    (233,367) .. controls (267.2,322.6) and (342.2,321.6) .. (373,367) ;
%Curve Lines [id:da46811598263406884] 
\draw [line width=2.25]    (93,367) .. controls (141.1,278.36) and (257.1,277.36) .. (303,367) ;
%Curve Lines [id:da2980037637421795] 
\draw [line width=2.25]    (163,367) .. controls (211.1,278.36) and (327.1,277.36) .. (373,367) ;
%Straight Lines [id:da14986120222470534] 
\draw    (198,437) -- (268,437) ;

% Text Node
\draw (191,445) node [anchor=north west][inner sep=0.75pt]  [color={rgb, 255:red, 0; green, 0; blue, 0 }  ,opacity=1 ] [align=left] {$\displaystyle u$};
% Text Node
\draw (263,445) node [anchor=north west][inner sep=0.75pt]  [color={rgb, 255:red, 0; green, 0; blue, 0 }  ,opacity=1 ] [align=left] {$\displaystyle v$};
% Text Node
\draw (66,362) node [anchor=north west][inner sep=0.75pt]  [color={rgb, 255:red, 0; green, 0; blue, 0 }  ,opacity=1 ] [align=left] {$\displaystyle w_{4}$};
% Text Node
\draw (171,362) node [anchor=north west][inner sep=0.75pt]  [color={rgb, 255:red, 0; green, 0; blue, 0 }  ,opacity=1 ] [align=left] {$\displaystyle w_{1}$};
% Text Node
\draw (241,362) node [anchor=north west][inner sep=0.75pt]  [color={rgb, 255:red, 0; green, 0; blue, 0 }  ,opacity=1 ] [align=left] {$\displaystyle w_{2}$};
% Text Node
\draw (310,362) node [anchor=north west][inner sep=0.75pt]  [color={rgb, 255:red, 0; green, 0; blue, 0 }  ,opacity=1 ] [align=left] {$\displaystyle w_{3}$};
% Text Node
\draw (380,362) node [anchor=north west][inner sep=0.75pt]  [color={rgb, 255:red, 0; green, 0; blue, 0 }  ,opacity=1 ] [align=left] {$\displaystyle x_{4}$};

\end{tikzpicture}
\caption{$uv$ sees the $11$ bold edges in $G-u$ by $G$.}
\label{fig1}
    \end{figure}

    To show the structure of the edges incident to $uw_1$ more clearly, Figure~\ref{fig1} is redrawn as Figure~\ref{fig2}. As shown in Figure~\ref{fig2}, $uw_1$ already sees $10$ edges in $G-u$ by $G$.

    \begin{figure}[H]
        \tikzset{every picture/.style={line width=0.75pt}} %set default line width to 0.75pt        

\begin{tikzpicture}[x=0.75pt,y=0.75pt,yscale=-1,xscale=1]
%uncomment if require: \path (0,1905); %set diagram left start at 0, and has height of 1905

%Shape: Circle [id:dp9895045453770652] 
\draw  [fill={rgb, 255:red, 0; green, 0; blue, 0 }  ,fill opacity=1 ] (540.5,444) .. controls (540.5,440.96) and (542.96,438.5) .. (546,438.5) .. controls (549.04,438.5) and (551.5,440.96) .. (551.5,444) .. controls (551.5,447.04) and (549.04,449.5) .. (546,449.5) .. controls (542.96,449.5) and (540.5,447.04) .. (540.5,444) -- cycle ;
%Shape: Circle [id:dp017705694980773234] 
\draw  [fill={rgb, 255:red, 0; green, 0; blue, 0 }  ,fill opacity=1 ] (610.5,444) .. controls (610.5,440.96) and (612.96,438.5) .. (616,438.5) .. controls (619.04,438.5) and (621.5,440.96) .. (621.5,444) .. controls (621.5,447.04) and (619.04,449.5) .. (616,449.5) .. controls (612.96,449.5) and (610.5,447.04) .. (610.5,444) -- cycle ;
%Shape: Circle [id:dp8969327372520656] 
\draw  [fill={rgb, 255:red, 0; green, 0; blue, 0 }  ,fill opacity=1 ] (435.5,374) .. controls (435.5,370.96) and (437.96,368.5) .. (441,368.5) .. controls (444.04,368.5) and (446.5,370.96) .. (446.5,374) .. controls (446.5,377.04) and (444.04,379.5) .. (441,379.5) .. controls (437.96,379.5) and (435.5,377.04) .. (435.5,374) -- cycle ;
%Shape: Circle [id:dp6687764817042279] 
\draw  [fill={rgb, 255:red, 0; green, 0; blue, 0 }  ,fill opacity=1 ] (505.5,374) .. controls (505.5,370.96) and (507.96,368.5) .. (511,368.5) .. controls (514.04,368.5) and (516.5,370.96) .. (516.5,374) .. controls (516.5,377.04) and (514.04,379.5) .. (511,379.5) .. controls (507.96,379.5) and (505.5,377.04) .. (505.5,374) -- cycle ;
%Shape: Circle [id:dp4826387865223931] 
\draw  [fill={rgb, 255:red, 0; green, 0; blue, 0 }  ,fill opacity=1 ] (575.5,374) .. controls (575.5,370.96) and (577.96,368.5) .. (581,368.5) .. controls (584.04,368.5) and (586.5,370.96) .. (586.5,374) .. controls (586.5,377.04) and (584.04,379.5) .. (581,379.5) .. controls (577.96,379.5) and (575.5,377.04) .. (575.5,374) -- cycle ;
%Shape: Circle [id:dp631545401332156] 
\draw  [fill={rgb, 255:red, 0; green, 0; blue, 0 }  ,fill opacity=1 ] (645.5,374) .. controls (645.5,370.96) and (647.96,368.5) .. (651,368.5) .. controls (654.04,368.5) and (656.5,370.96) .. (656.5,374) .. controls (656.5,377.04) and (654.04,379.5) .. (651,379.5) .. controls (647.96,379.5) and (645.5,377.04) .. (645.5,374) -- cycle ;
%Shape: Circle [id:dp08706549516335038] 
\draw  [fill={rgb, 255:red, 0; green, 0; blue, 0 }  ,fill opacity=1 ] (715.5,374) .. controls (715.5,370.96) and (717.96,368.5) .. (721,368.5) .. controls (724.04,368.5) and (726.5,370.96) .. (726.5,374) .. controls (726.5,377.04) and (724.04,379.5) .. (721,379.5) .. controls (717.96,379.5) and (715.5,377.04) .. (715.5,374) -- cycle ;
%Straight Lines [id:da6894079576592486] 
\draw    (441,374) -- (546,444) ;
%Straight Lines [id:da41622368277465793] 
\draw    (511,374) -- (546,444) ;
%Straight Lines [id:da7271051405998147] 
\draw    (581,374) -- (546,444) ;
%Straight Lines [id:da26968260959945467] 
\draw    (651,374) -- (546,444) ;
%Straight Lines [id:da26844821040432376] 
\draw [line width=2.25]    (511,374) -- (616,444) ;
%Straight Lines [id:da3549133162696376] 
\draw [line width=2.25]    (581,374) -- (616,444) ;
%Straight Lines [id:da12746596778882735] 
\draw [line width=2.25]    (721,374) -- (616,444) ;
%Straight Lines [id:da05923583609555194] 
\draw [line width=2.25]    (651,374) -- (616,444) ;
%Curve Lines [id:da06160818086099584] 
\draw [line width=2.25]    (441,374) .. controls (462.2,358.6) and (487.2,357.6) .. (511,374) ;
%Curve Lines [id:da7051387187611974] 
\draw [line width=2.25]    (511,374) .. controls (532.2,358.6) and (557.2,357.6) .. (581,374) ;
%Curve Lines [id:da6921480574115034] 
\draw [line width=2.25]    (581,374) .. controls (615.2,329.6) and (690.2,328.6) .. (721,374) ;
%Curve Lines [id:da19015067685172526] 
\draw [line width=2.25]    (441,374) .. controls (489.1,285.36) and (605.1,284.36) .. (651,374) ;
%Curve Lines [id:da4266285532122517] 
\draw [line width=2.25]    (511,374) .. controls (559.1,285.36) and (675.1,284.36) .. (721,374) ;
%Straight Lines [id:da42138818145526813] 
\draw    (546,444) -- (616,444) ;
%Curve Lines [id:da042912131139415055] 
\draw [line width=0.75]    (441,374) .. controls (490.33,236) and (675.33,235) .. (721,374) ;
%Curve Lines [id:da4008506946196918] 
\draw [line width=2.25]    (581,374) .. controls (602.2,358.6) and (627.2,357.6) .. (651,374) ;

% Text Node
\draw (539,452) node [anchor=north west][inner sep=0.75pt]  [color={rgb, 255:red, 0; green, 0; blue, 0 }  ,opacity=1 ] [align=left] {$\displaystyle u$};
% Text Node
\draw (611,452) node [anchor=north west][inner sep=0.75pt]  [color={rgb, 255:red, 0; green, 0; blue, 0 }  ,opacity=1 ] [align=left] {$\displaystyle w_{1}$};
% Text Node
\draw (414,369) node [anchor=north west][inner sep=0.75pt]  [color={rgb, 255:red, 0; green, 0; blue, 0 }  ,opacity=1 ] [align=left] {$\displaystyle w_{3}$};
% Text Node
\draw (519,369) node [anchor=north west][inner sep=0.75pt]  [color={rgb, 255:red, 0; green, 0; blue, 0 }  ,opacity=1 ] [align=left] {$\displaystyle v$};
% Text Node
\draw (589,369) node [anchor=north west][inner sep=0.75pt]  [color={rgb, 255:red, 0; green, 0; blue, 0 }  ,opacity=1 ] [align=left] {$\displaystyle w_{2}$};
% Text Node
\draw (658,369) node [anchor=north west][inner sep=0.75pt]  [color={rgb, 255:red, 0; green, 0; blue, 0 }  ,opacity=1 ] [align=left] {$\displaystyle w_{4}$};
% Text Node
\draw (728,369) node [anchor=north west][inner sep=0.75pt]  [color={rgb, 255:red, 0; green, 0; blue, 0 }  ,opacity=1 ] [align=left] {$\displaystyle x_{4}$};

\end{tikzpicture}
\caption{$uw_1$ sees the $10$ bold edges in $G-u$ by $G$.}
\label{fig2}
    \end{figure}
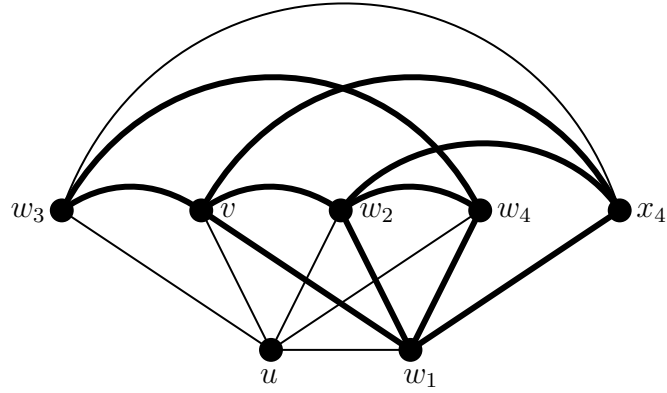

    In this subgraph of $G$, observe that 
    \begin{itemize}
        \item $w_3,\ w_4,\ x_4$ have degree $4$, and all other vertices have degree $5$;
        \item $w_3$ is adjacent to $w_4$ and $x_4$.
    \end{itemize}

    If $w_4$ and $x_4$ are not adjacent, then $uw_1$ only sees $10$ edges in $G-u$ by $G$, which means $\ell_f(uw_1)\ge 5$, a contradiction. So $w_4$ and $x_4$ are adjacent. In the resulting subgraph, all vertices except $w_3$ have degree $5$. As $G$ is a $5$-regular graph, we know that $w_3$ has another neighbor, say $y$, that is not in Figure~\ref{fig2}. We know that $w_3$ and $y$ do not have any common neighbors, because the other four neighbors of $w_3$ are $u,\,v,\,w_4,\,x_4$, each of which already has five neighbors. Now we have determined a subgraph of $G$ shown in Figure~\ref{fig3}.

    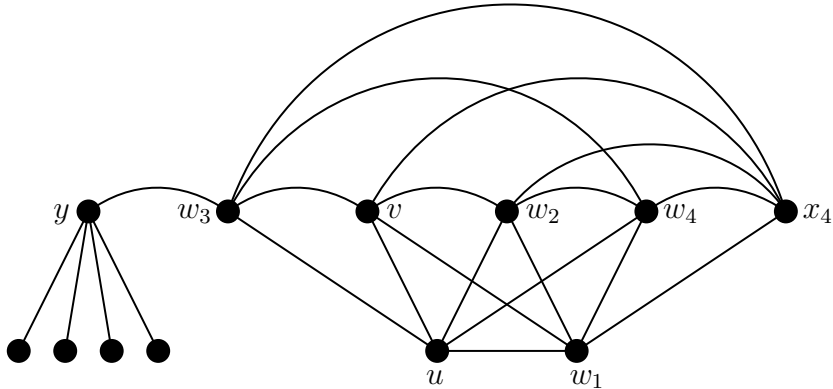
\begin{figure}[H]
        \tikzset{every picture/.style={line width=0.75pt}} %set default line width to 0.75pt        

\begin{tikzpicture}[x=0.75pt,y=0.75pt,yscale=-1,xscale=1]
%uncomment if require: \path (0,2276); %set diagram left start at 0, and has height of 2276

%Shape: Circle [id:dp884880717407953] 
\draw  [fill={rgb, 255:red, 0; green, 0; blue, 0 }  ,fill opacity=1 ] (350.5,708) .. controls (350.5,704.96) and (352.96,702.5) .. (356,702.5) .. controls (359.04,702.5) and (361.5,704.96) .. (361.5,708) .. controls (361.5,711.04) and (359.04,713.5) .. (356,713.5) .. controls (352.96,713.5) and (350.5,711.04) .. (350.5,708) -- cycle ;
%Shape: Circle [id:dp8049459315973136] 
\draw  [fill={rgb, 255:red, 0; green, 0; blue, 0 }  ,fill opacity=1 ] (420.5,708) .. controls (420.5,704.96) and (422.96,702.5) .. (426,702.5) .. controls (429.04,702.5) and (431.5,704.96) .. (431.5,708) .. controls (431.5,711.04) and (429.04,713.5) .. (426,713.5) .. controls (422.96,713.5) and (420.5,711.04) .. (420.5,708) -- cycle ;
%Shape: Circle [id:dp7853334073918379] 
\draw  [fill={rgb, 255:red, 0; green, 0; blue, 0 }  ,fill opacity=1 ] (245.5,638) .. controls (245.5,634.96) and (247.96,632.5) .. (251,632.5) .. controls (254.04,632.5) and (256.5,634.96) .. (256.5,638) .. controls (256.5,641.04) and (254.04,643.5) .. (251,643.5) .. controls (247.96,643.5) and (245.5,641.04) .. (245.5,638) -- cycle ;
%Shape: Circle [id:dp8235752835955217] 
\draw  [fill={rgb, 255:red, 0; green, 0; blue, 0 }  ,fill opacity=1 ] (315.5,638) .. controls (315.5,634.96) and (317.96,632.5) .. (321,632.5) .. controls (324.04,632.5) and (326.5,634.96) .. (326.5,638) .. controls (326.5,641.04) and (324.04,643.5) .. (321,643.5) .. controls (317.96,643.5) and (315.5,641.04) .. (315.5,638) -- cycle ;
%Shape: Circle [id:dp4373980631620419] 
\draw  [fill={rgb, 255:red, 0; green, 0; blue, 0 }  ,fill opacity=1 ] (385.5,638) .. controls (385.5,634.96) and (387.96,632.5) .. (391,632.5) .. controls (394.04,632.5) and (396.5,634.96) .. (396.5,638) .. controls (396.5,641.04) and (394.04,643.5) .. (391,643.5) .. controls (387.96,643.5) and (385.5,641.04) .. (385.5,638) -- cycle ;
%Shape: Circle [id:dp7779054296568054] 
\draw  [fill={rgb, 255:red, 0; green, 0; blue, 0 }  ,fill opacity=1 ] (455.5,638) .. controls (455.5,634.96) and (457.96,632.5) .. (461,632.5) .. controls (464.04,632.5) and (466.5,634.96) .. (466.5,638) .. controls (466.5,641.04) and (464.04,643.5) .. (461,643.5) .. controls (457.96,643.5) and (455.5,641.04) .. (455.5,638) -- cycle ;
%Shape: Circle [id:dp22251640994224575] 
\draw  [fill={rgb, 255:red, 0; green, 0; blue, 0 }  ,fill opacity=1 ] (525.5,638) .. controls (525.5,634.96) and (527.96,632.5) .. (531,632.5) .. controls (534.04,632.5) and (536.5,634.96) .. (536.5,638) .. controls (536.5,641.04) and (534.04,643.5) .. (531,643.5) .. controls (527.96,643.5) and (525.5,641.04) .. (525.5,638) -- cycle ;
%Straight Lines [id:da5116122899474588] 
\draw [line width=0.75]    (251,638) -- (356,708) ;
%Straight Lines [id:da5034662145354167] 
\draw [line width=0.75]    (321,638) -- (356,708) ;
%Straight Lines [id:da0463865912490834] 
\draw [line width=0.75]    (391,638) -- (356,708) ;
%Straight Lines [id:da9897062135275843] 
\draw [line width=0.75]    (461,638) -- (356,708) ;
%Straight Lines [id:da6893794839311677] 
\draw [line width=0.75]    (321,638) -- (426,708) ;
%Straight Lines [id:da1718226041981893] 
\draw [line width=0.75]    (391,638) -- (426,708) ;
%Straight Lines [id:da9252306699604256] 
\draw [line width=0.75]    (531,638) -- (426,708) ;
%Straight Lines [id:da22889310051377632] 
\draw [line width=0.75]    (461,638) -- (426,708) ;
%Curve Lines [id:da748139807155767] 
\draw [line width=0.75]    (251,638) .. controls (272.2,622.6) and (297.2,621.6) .. (321,638) ;
%Curve Lines [id:da6128786690961019] 
\draw [line width=0.75]    (321,638) .. controls (342.2,622.6) and (367.2,621.6) .. (391,638) ;
%Curve Lines [id:da3273126436027488] 
\draw [line width=0.75]    (391,638) .. controls (425.2,593.6) and (500.2,592.6) .. (531,638) ;
%Curve Lines [id:da3032112749102023] 
\draw [line width=0.75]    (251,638) .. controls (299.1,549.36) and (415.1,548.36) .. (461,638) ;
%Curve Lines [id:da7060606045781787] 
\draw [line width=0.75]    (321,638) .. controls (369.1,549.36) and (485.1,548.36) .. (531,638) ;
%Straight Lines [id:da5146451980670814] 
\draw    (356,708) -- (426,708) ;
%Curve Lines [id:da3248398843395167] 
\draw [line width=0.75]    (251,638) .. controls (300.33,500) and (485.33,499) .. (531,638) ;
%Curve Lines [id:da5320267235913486] 
\draw [line width=0.75]    (391,638) .. controls (412.2,622.6) and (437.2,621.6) .. (461,638) ;
%Curve Lines [id:da5540233787766837] 
\draw [line width=0.75]    (181,638) .. controls (202.2,622.6) and (227.2,621.6) .. (251,638) ;
%Curve Lines [id:da2460654996095648] 
\draw [line width=0.75]    (461,638) .. controls (482.2,622.6) and (507.2,621.6) .. (531,638) ;
%Shape: Circle [id:dp8931005239617252] 
\draw  [fill={rgb, 255:red, 0; green, 0; blue, 0 }  ,fill opacity=1 ] (175.5,638) .. controls (175.5,634.96) and (177.96,632.5) .. (181,632.5) .. controls (184.04,632.5) and (186.5,634.96) .. (186.5,638) .. controls (186.5,641.04) and (184.04,643.5) .. (181,643.5) .. controls (177.96,643.5) and (175.5,641.04) .. (175.5,638) -- cycle ;
%Straight Lines [id:da32234049503383166] 
\draw [line width=0.75]    (181,638) -- (192.66,708) ;
%Straight Lines [id:da640784517992497] 
\draw [line width=0.75]    (181,638) -- (216,708) ;
%Straight Lines [id:da2292097616449782] 
\draw [line width=0.75]    (181,638) -- (146,708) ;
%Straight Lines [id:da35854778035817036] 
\draw [line width=0.75]    (181,638) -- (169.33,708) ;
%Shape: Circle [id:dp23638700280717573] 
\draw  [fill={rgb, 255:red, 0; green, 0; blue, 0 }  ,fill opacity=1 ] (163.83,708) .. controls (163.83,704.96) and (166.29,702.5) .. (169.33,702.5) .. controls (172.37,702.5) and (174.83,704.96) .. (174.83,708) .. controls (174.83,711.04) and (172.37,713.5) .. (169.33,713.5) .. controls (166.29,713.5) and (163.83,711.04) .. (163.83,708) -- cycle ;
%Shape: Circle [id:dp9563578371966294] 
\draw  [fill={rgb, 255:red, 0; green, 0; blue, 0 }  ,fill opacity=1 ] (140.5,708) .. controls (140.5,704.96) and (142.96,702.5) .. (146,702.5) .. controls (149.04,702.5) and (151.5,704.96) .. (151.5,708) .. controls (151.5,711.04) and (149.04,713.5) .. (146,713.5) .. controls (142.96,713.5) and (140.5,711.04) .. (140.5,708) -- cycle ;
%Shape: Circle [id:dp9542044857341715] 
\draw  [fill={rgb, 255:red, 0; green, 0; blue, 0 }  ,fill opacity=1 ] (210.5,708) .. controls (210.5,704.96) and (212.96,702.5) .. (216,702.5) .. controls (219.04,702.5) and (221.5,704.96) .. (221.5,708) .. controls (221.5,711.04) and (219.04,713.5) .. (216,713.5) .. controls (212.96,713.5) and (210.5,711.04) .. (210.5,708) -- cycle ;
%Shape: Circle [id:dp09249953783870801] 
\draw  [fill={rgb, 255:red, 0; green, 0; blue, 0 }  ,fill opacity=1 ] (187.16,708) .. controls (187.16,704.96) and (189.62,702.5) .. (192.66,702.5) .. controls (195.7,702.5) and (198.16,704.96) .. (198.16,708) .. controls (198.16,711.04) and (195.7,713.5) .. (192.66,713.5) .. controls (189.62,713.5) and (187.16,711.04) .. (187.16,708) -- cycle ;

% Text Node
\draw (349,716) node [anchor=north west][inner sep=0.75pt]  [color={rgb, 255:red, 0; green, 0; blue, 0 }  ,opacity=1 ] [align=left] {$\displaystyle u$};
% Text Node
\draw (421,716) node [anchor=north west][inner sep=0.75pt]  [color={rgb, 255:red, 0; green, 0; blue, 0 }  ,opacity=1 ] [align=left] {$\displaystyle w_{1}$};
% Text Node
\draw (224,633) node [anchor=north west][inner sep=0.75pt]  [color={rgb, 255:red, 0; green, 0; blue, 0 }  ,opacity=1 ] [align=left] {$\displaystyle w_{3}$};
% Text Node
\draw (329,633) node [anchor=north west][inner sep=0.75pt]  [color={rgb, 255:red, 0; green, 0; blue, 0 }  ,opacity=1 ] [align=left] {$\displaystyle v$};
% Text Node
\draw (399,633) node [anchor=north west][inner sep=0.75pt]  [color={rgb, 255:red, 0; green, 0; blue, 0 }  ,opacity=1 ] [align=left] {$\displaystyle w_{2}$};
% Text Node
\draw (468,633) node [anchor=north west][inner sep=0.75pt]  [color={rgb, 255:red, 0; green, 0; blue, 0 }  ,opacity=1 ] [align=left] {$\displaystyle w_{4}$};
% Text Node
\draw (538,633) node [anchor=north west][inner sep=0.75pt]  [color={rgb, 255:red, 0; green, 0; blue, 0 }  ,opacity=1 ] [align=left] {$\displaystyle x_{4}$};
% Text Node
\draw (162,633) node [anchor=north west][inner sep=0.75pt]  [color={rgb, 255:red, 0; green, 0; blue, 0 }  ,opacity=1 ] [align=left] {$\displaystyle y$};

\end{tikzpicture}
\caption{A subgraph of $G$.}
\label{fig3}
    \end{figure}

    By the minimality of $G$, we know that $G-w_3$ is $15$-D-colorable. Let $f':E(G-w_3)\longrightarrow \ints{15}$ be a $15$-D-coloring of $G-w_3$, so $f'$ is a partial coloring of $G$. As $w_3$ and $y$ do not have any common neighbors, $w_3 y$ is not in any diamond subgraphs of $G$, which means $w_3 y$ only sees four edges in $G-w_3$ (the four edges incident to $y$) by $G$. Now, for the edges incident to $w_3$, we have $\ell_{f'}(w_3 y)\ge 15-4=11$ and $\ell_{f'}(w_3 u),\,\ell_{f'}(w_3 v),\,\ell_{f'}(w_3 w_4),\,\ell_{f'}(w_3 x_4)\ge 4$. So, by Lemma~\ref{extension}, we can extend $f'$ to a $15$-D-coloring of $G$, contradicting the assumption that $G$ is a counterexample.

    \textbf{Case 3.} There are exactly two edges with both endpoints in $Z=\{w_1,\,w_2,\,w_3\}$.

    By symmetry, we may assume that $w_1 w_2$ and $w_2 w_3$ are the two edges with both endpoints in $Z$. Then, there are five edges each with one endpoint in $\{w_1,\,w_2,\,w_3\}$ and the other endpoint in $\{w_4,\,x_4\}$. We know that $w_1$ already has three neighbors, $w_2$ already has four neighbors, and $w_3$ already has three neighbors, so 
    \begin{itemize}
        \item $w_1$ must be adjacent to both $w_4$ and $x_4$;
        \item $w_2$ must be adjacent to either $w_4$ or $x_4$;
        \item $w_3$ must be adjacent to both $w_4$ and $x_4$.
    \end{itemize}

    \textbf{Subcase 3.1.} $w_2$ is adjacent to $w_4$, which means $\{w_1 w_4,\,w_1 x_4,\,w_2 w_4,\,w_3 w_4,\,w_3 x_4\}\subseteq E(G)$.
    
    So far, we have determined a subgraph of $G$ shown in Figure~\ref{fig4}.

    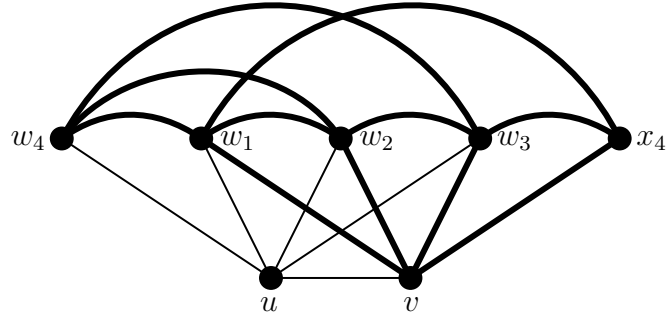
\begin{figure}[H]
        \tikzset{every picture/.style={line width=0.75pt}} %set default line width to 0.75pt        

\begin{tikzpicture}[x=0.75pt,y=0.75pt,yscale=-1,xscale=1]
%uncomment if require: \path (0,2276); %set diagram left start at 0, and has height of 2276

%Shape: Circle [id:dp6794061471219252] 
\draw  [fill={rgb, 255:red, 0; green, 0; blue, 0 }  ,fill opacity=1 ] (165.5,905.45) .. controls (165.5,902.42) and (167.96,899.95) .. (171,899.95) .. controls (174.04,899.95) and (176.5,902.42) .. (176.5,905.45) .. controls (176.5,908.49) and (174.04,910.95) .. (171,910.95) .. controls (167.96,910.95) and (165.5,908.49) .. (165.5,905.45) -- cycle ;
%Shape: Circle [id:dp19648445424040117] 
\draw  [fill={rgb, 255:red, 0; green, 0; blue, 0 }  ,fill opacity=1 ] (235.5,905.45) .. controls (235.5,902.42) and (237.96,899.95) .. (241,899.95) .. controls (244.04,899.95) and (246.5,902.42) .. (246.5,905.45) .. controls (246.5,908.49) and (244.04,910.95) .. (241,910.95) .. controls (237.96,910.95) and (235.5,908.49) .. (235.5,905.45) -- cycle ;
%Shape: Circle [id:dp9791136415874764] 
\draw  [fill={rgb, 255:red, 0; green, 0; blue, 0 }  ,fill opacity=1 ] (60.5,835.45) .. controls (60.5,832.42) and (62.96,829.95) .. (66,829.95) .. controls (69.04,829.95) and (71.5,832.42) .. (71.5,835.45) .. controls (71.5,838.49) and (69.04,840.95) .. (66,840.95) .. controls (62.96,840.95) and (60.5,838.49) .. (60.5,835.45) -- cycle ;
%Shape: Circle [id:dp7616756092412409] 
\draw  [fill={rgb, 255:red, 0; green, 0; blue, 0 }  ,fill opacity=1 ] (130.5,835.45) .. controls (130.5,832.42) and (132.96,829.95) .. (136,829.95) .. controls (139.04,829.95) and (141.5,832.42) .. (141.5,835.45) .. controls (141.5,838.49) and (139.04,840.95) .. (136,840.95) .. controls (132.96,840.95) and (130.5,838.49) .. (130.5,835.45) -- cycle ;
%Shape: Circle [id:dp2750154137350903] 
\draw  [fill={rgb, 255:red, 0; green, 0; blue, 0 }  ,fill opacity=1 ] (200.5,835.45) .. controls (200.5,832.42) and (202.96,829.95) .. (206,829.95) .. controls (209.04,829.95) and (211.5,832.42) .. (211.5,835.45) .. controls (211.5,838.49) and (209.04,840.95) .. (206,840.95) .. controls (202.96,840.95) and (200.5,838.49) .. (200.5,835.45) -- cycle ;
%Shape: Circle [id:dp6951959499572] 
\draw  [fill={rgb, 255:red, 0; green, 0; blue, 0 }  ,fill opacity=1 ] (270.5,835.45) .. controls (270.5,832.42) and (272.96,829.95) .. (276,829.95) .. controls (279.04,829.95) and (281.5,832.42) .. (281.5,835.45) .. controls (281.5,838.49) and (279.04,840.95) .. (276,840.95) .. controls (272.96,840.95) and (270.5,838.49) .. (270.5,835.45) -- cycle ;
%Shape: Circle [id:dp3456804655384741] 
\draw  [fill={rgb, 255:red, 0; green, 0; blue, 0 }  ,fill opacity=1 ] (340.5,835.45) .. controls (340.5,832.42) and (342.96,829.95) .. (346,829.95) .. controls (349.04,829.95) and (351.5,832.42) .. (351.5,835.45) .. controls (351.5,838.49) and (349.04,840.95) .. (346,840.95) .. controls (342.96,840.95) and (340.5,838.49) .. (340.5,835.45) -- cycle ;
%Straight Lines [id:da029887161152679087] 
\draw    (66,835.45) -- (171,905.45) ;
%Straight Lines [id:da9528852457956299] 
\draw    (136,835.45) -- (171,905.45) ;
%Straight Lines [id:da779634515195542] 
\draw    (206,835.45) -- (171,905.45) ;
%Straight Lines [id:da2955120317583616] 
\draw    (276,835.45) -- (171,905.45) ;
%Straight Lines [id:da011956158557505536] 
\draw [line width=2.25]    (136,835.45) -- (241,905.45) ;
%Straight Lines [id:da47095277956284454] 
\draw [line width=2.25]    (206,835.45) -- (241,905.45) ;
%Straight Lines [id:da11566095427481926] 
\draw [line width=2.25]    (346,835.45) -- (241,905.45) ;
%Straight Lines [id:da9844420463625737] 
\draw [line width=2.25]    (276,835.45) -- (241,905.45) ;
%Curve Lines [id:da506142663227691] 
\draw [line width=2.25]    (66,835.45) .. controls (87.2,820.05) and (112.2,819.05) .. (136,835.45) ;
%Curve Lines [id:da8802029659144277] 
\draw [line width=2.25]    (136,835.45) .. controls (157.2,820.05) and (182.2,819.05) .. (206,835.45) ;
%Curve Lines [id:da7658342850827086] 
\draw [line width=2.25]    (276,835.45) .. controls (297.2,820.05) and (322.2,819.05) .. (346,835.45) ;
%Curve Lines [id:da08407809832643154] 
\draw [line width=2.25]    (66,835.45) .. controls (100.2,791.05) and (175.2,790.05) .. (206,835.45) ;
%Curve Lines [id:da41884493210966556] 
\draw [line width=2.25]    (66,835.45) .. controls (114.1,746.82) and (230.1,745.82) .. (276,835.45) ;
%Curve Lines [id:da8450274471345318] 
\draw [line width=2.25]    (136,835.45) .. controls (184.1,746.82) and (300.1,745.82) .. (346,835.45) ;
%Straight Lines [id:da32780610412117506] 
\draw    (171,905.45) -- (241,905.45) ;
%Curve Lines [id:da7660933580579489] 
\draw [line width=2.25]    (206,835.45) .. controls (227.2,820.05) and (252.2,819.05) .. (276,835.45) ;

% Text Node
\draw (164,913.45) node [anchor=north west][inner sep=0.75pt]  [color={rgb, 255:red, 0; green, 0; blue, 0 }  ,opacity=1 ] [align=left] {$\displaystyle u$};
% Text Node
\draw (236,913.45) node [anchor=north west][inner sep=0.75pt]  [color={rgb, 255:red, 0; green, 0; blue, 0 }  ,opacity=1 ] [align=left] {$\displaystyle v$};
% Text Node
\draw (39,830.45) node [anchor=north west][inner sep=0.75pt]  [color={rgb, 255:red, 0; green, 0; blue, 0 }  ,opacity=1 ] [align=left] {$\displaystyle w_{4}$};
% Text Node
\draw (144,830.45) node [anchor=north west][inner sep=0.75pt]  [color={rgb, 255:red, 0; green, 0; blue, 0 }  ,opacity=1 ] [align=left] {$\displaystyle w_{1}$};
% Text Node
\draw (214,830.45) node [anchor=north west][inner sep=0.75pt]  [color={rgb, 255:red, 0; green, 0; blue, 0 }  ,opacity=1 ] [align=left] {$\displaystyle w_{2}$};
% Text Node
\draw (283,830.45) node [anchor=north west][inner sep=0.75pt]  [color={rgb, 255:red, 0; green, 0; blue, 0 }  ,opacity=1 ] [align=left] {$\displaystyle w_{3}$};
% Text Node
\draw (353,830.45) node [anchor=north west][inner sep=0.75pt]  [color={rgb, 255:red, 0; green, 0; blue, 0 }  ,opacity=1 ] [align=left] {$\displaystyle x_{4}$};

\end{tikzpicture}
\caption{$uv$ sees the $11$ bold edges in $G-u$ by $G$.}
\label{fig4}
    \end{figure}

    To show the structure of the edges incident to $uw_1$ more clearly, Figure~\ref{fig4} is redrawn as Figure~\ref{fig5}. As shown in Figure~\ref{fig5}, $uw_1$ already sees $10$ edges in $G-u$ by $G$.

    \begin{figure}[H]
        \tikzset{every picture/.style={line width=0.75pt}} %set default line width to 0.75pt        

\begin{tikzpicture}[x=0.75pt,y=0.75pt,yscale=-1,xscale=1]
%uncomment if require: \path (0,2276); %set diagram left start at 0, and has height of 2276

%Shape: Circle [id:dp12159989267229343] 
\draw  [fill={rgb, 255:red, 0; green, 0; blue, 0 }  ,fill opacity=1 ] (541.5,929) .. controls (541.5,925.96) and (543.96,923.5) .. (547,923.5) .. controls (550.04,923.5) and (552.5,925.96) .. (552.5,929) .. controls (552.5,932.04) and (550.04,934.5) .. (547,934.5) .. controls (543.96,934.5) and (541.5,932.04) .. (541.5,929) -- cycle ;
%Shape: Circle [id:dp6963645741766236] 
\draw  [fill={rgb, 255:red, 0; green, 0; blue, 0 }  ,fill opacity=1 ] (611.5,929) .. controls (611.5,925.96) and (613.96,923.5) .. (617,923.5) .. controls (620.04,923.5) and (622.5,925.96) .. (622.5,929) .. controls (622.5,932.04) and (620.04,934.5) .. (617,934.5) .. controls (613.96,934.5) and (611.5,932.04) .. (611.5,929) -- cycle ;
%Shape: Circle [id:dp24773439006922582] 
\draw  [fill={rgb, 255:red, 0; green, 0; blue, 0 }  ,fill opacity=1 ] (436.5,859) .. controls (436.5,855.96) and (438.96,853.5) .. (442,853.5) .. controls (445.04,853.5) and (447.5,855.96) .. (447.5,859) .. controls (447.5,862.04) and (445.04,864.5) .. (442,864.5) .. controls (438.96,864.5) and (436.5,862.04) .. (436.5,859) -- cycle ;
%Shape: Circle [id:dp9503265143679341] 
\draw  [fill={rgb, 255:red, 0; green, 0; blue, 0 }  ,fill opacity=1 ] (506.5,859) .. controls (506.5,855.96) and (508.96,853.5) .. (512,853.5) .. controls (515.04,853.5) and (517.5,855.96) .. (517.5,859) .. controls (517.5,862.04) and (515.04,864.5) .. (512,864.5) .. controls (508.96,864.5) and (506.5,862.04) .. (506.5,859) -- cycle ;
%Shape: Circle [id:dp8945163409696586] 
\draw  [fill={rgb, 255:red, 0; green, 0; blue, 0 }  ,fill opacity=1 ] (576.5,859) .. controls (576.5,855.96) and (578.96,853.5) .. (582,853.5) .. controls (585.04,853.5) and (587.5,855.96) .. (587.5,859) .. controls (587.5,862.04) and (585.04,864.5) .. (582,864.5) .. controls (578.96,864.5) and (576.5,862.04) .. (576.5,859) -- cycle ;
%Shape: Circle [id:dp2719237749011424] 
\draw  [fill={rgb, 255:red, 0; green, 0; blue, 0 }  ,fill opacity=1 ] (646.5,859) .. controls (646.5,855.96) and (648.96,853.5) .. (652,853.5) .. controls (655.04,853.5) and (657.5,855.96) .. (657.5,859) .. controls (657.5,862.04) and (655.04,864.5) .. (652,864.5) .. controls (648.96,864.5) and (646.5,862.04) .. (646.5,859) -- cycle ;
%Shape: Circle [id:dp1886896114717721] 
\draw  [fill={rgb, 255:red, 0; green, 0; blue, 0 }  ,fill opacity=1 ] (716.5,859) .. controls (716.5,855.96) and (718.96,853.5) .. (722,853.5) .. controls (725.04,853.5) and (727.5,855.96) .. (727.5,859) .. controls (727.5,862.04) and (725.04,864.5) .. (722,864.5) .. controls (718.96,864.5) and (716.5,862.04) .. (716.5,859) -- cycle ;
%Straight Lines [id:da8764559227534371] 
\draw    (442,859) -- (547,929) ;
%Straight Lines [id:da8459192622768932] 
\draw    (512,859) -- (547,929) ;
%Straight Lines [id:da8046367285821258] 
\draw    (582,859) -- (547,929) ;
%Straight Lines [id:da2160905890136886] 
\draw    (652,859) -- (547,929) ;
%Straight Lines [id:da8770437345024903] 
\draw [line width=2.25]    (512,859) -- (617,929) ;
%Straight Lines [id:da9577833657142072] 
\draw [line width=2.25]    (582,859) -- (617,929) ;
%Straight Lines [id:da3261226708116586] 
\draw [line width=2.25]    (722,859) -- (617,929) ;
%Straight Lines [id:da4722556896075166] 
\draw [line width=2.25]    (652,859) -- (617,929) ;
%Curve Lines [id:da9618331423000243] 
\draw [line width=2.25]    (442,859) .. controls (463.2,843.6) and (488.2,842.6) .. (512,859) ;
%Curve Lines [id:da7715682724068258] 
\draw [line width=2.25]    (512,859) .. controls (533.2,843.6) and (558.2,842.6) .. (582,859) ;
%Curve Lines [id:da9031022674881148] 
\draw [line width=2.25]    (442,859) .. controls (476.2,814.6) and (551.2,813.6) .. (582,859) ;
%Curve Lines [id:da4452672233769779] 
\draw [line width=2.25]    (442,859) .. controls (490.1,770.36) and (606.1,769.36) .. (652,859) ;
%Curve Lines [id:da2992572496576097] 
\draw [line width=2.25]    (512,859) .. controls (560.1,770.36) and (676.1,769.36) .. (722,859) ;
%Straight Lines [id:da560108358047234] 
\draw    (547,929) -- (617,929) ;
%Curve Lines [id:da44026493365173913] 
\draw [line width=0.75]    (442,859) .. controls (491.33,721) and (676.33,720) .. (722,859) ;
%Curve Lines [id:da7366313809732198] 
\draw [line width=2.25]    (582,859) .. controls (603.2,843.6) and (628.2,842.6) .. (652,859) ;

% Text Node
\draw (540,937) node [anchor=north west][inner sep=0.75pt]  [color={rgb, 255:red, 0; green, 0; blue, 0 }  ,opacity=1 ] [align=left] {$\displaystyle u$};
% Text Node
\draw (612,937) node [anchor=north west][inner sep=0.75pt]  [color={rgb, 255:red, 0; green, 0; blue, 0 }  ,opacity=1 ] [align=left] {$\displaystyle w_{1}$};
% Text Node
\draw (415,854) node [anchor=north west][inner sep=0.75pt]  [color={rgb, 255:red, 0; green, 0; blue, 0 }  ,opacity=1 ] [align=left] {$\displaystyle w_{3}$};
% Text Node
\draw (520,854) node [anchor=north west][inner sep=0.75pt]  [color={rgb, 255:red, 0; green, 0; blue, 0 }  ,opacity=1 ] [align=left] {$\displaystyle v$};
% Text Node
\draw (590,854) node [anchor=north west][inner sep=0.75pt]  [color={rgb, 255:red, 0; green, 0; blue, 0 }  ,opacity=1 ] [align=left] {$\displaystyle w_{2}$};
% Text Node
\draw (659,854) node [anchor=north west][inner sep=0.75pt]  [color={rgb, 255:red, 0; green, 0; blue, 0 }  ,opacity=1 ] [align=left] {$\displaystyle w_{4}$};
% Text Node
\draw (729,854) node [anchor=north west][inner sep=0.75pt]  [color={rgb, 255:red, 0; green, 0; blue, 0 }  ,opacity=1 ] [align=left] {$\displaystyle x_{4}$};

\end{tikzpicture}
\caption{$uw_1$ sees the $10$ bold edges in $G-u$ by $G$.}
\label{fig5}
    \end{figure}
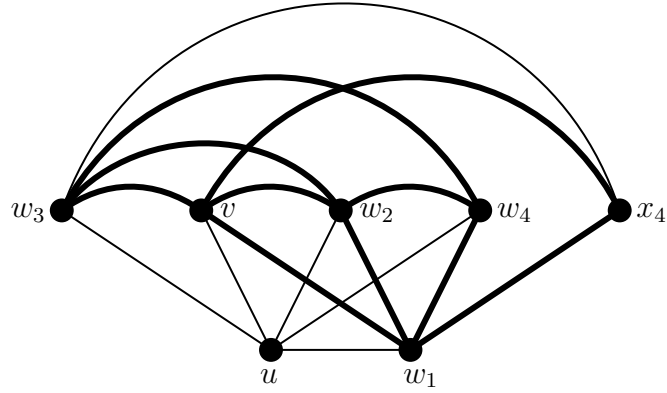

    In this subgraph of $G$, observe that $w_4$ has degree $4$, $x_4$ has degree $3$, and all other vertices have degree $5$.

    If $w_4$ and $x_4$ are not adjacent, then $uw_1$ only sees $10$ edges in $G-u$ by $G$, which means $\ell_f(uw_1)\ge 5$, a contradiction. So $w_4$ and $x_4$ are adjacent. In the resulting subgraph, all vertices except $x_4$ have degree $5$. As $G$ is a $5$-regular graph, we know that $x_4$ has another neighbor, say $y$, that is not in Figure~\ref{fig5}. We know that $x_4$ and $y$ do not have any common neighbors, because the other four neighbors of $x_4$ are $v,\,w_1,\,w_3,\,w_4$, each of which already has five neighbors. Now we have determined a subgraph of $G$, which is isomorphic to the one in Figure~\ref{fig3}. Similar to Case 2, we can extend a $15$-D-coloring of $G-x_4$ to a $15$-D-coloring of $G$, contradicting the assumption that $G$ is a counterexample.
    
    \textbf{Subcase 3.2.} $w_2$ is adjacent to $x_4$, which means $\{w_1 w_4,\,w_1 x_4,\,w_2 x_4,\,w_3 w_4,\,w_3 x_4\}\subseteq E(G)$.

    The common neighbors of $u$ and $w_1$ are $v,\,w_2,\,w_4$. The union of their neighborhoods, with $u$ and $w_1$ excluded, is $\{v,\,w_2,\,w_3,\,w_4,\,x_4\}$. By the characterization in Section~2, if $w_4x_4\notin E(G)$, the edges disjoint from $uw_1$ that it sees by $G$ are precisely $vw_2,\,vw_3,\,vx_4,\,w_2w_3,\,w_2x_4,\,w_3w_4$. Together with the four edges incident to $w_1$ in $G-u$, these give $\ell_f(uw_1)\ge 15-10=5$, a contradiction. Hence $w_4x_4\in E(G)$.

    In the subgraph determined so far, $w_4$ has neighbors $u,\,w_1,\,w_3,\,x_4$, and the other six vertices have degree $5$. Since $G$ is $5$-regular, $w_4$ has a fifth neighbor $y$ outside this subgraph. Each of $u,\,w_1,\,w_3,\,x_4$ already has degree $5$ in the subgraph, so $w_4$ and $y$ have no common neighbor. Thus $w_4 y$ does not lie in a diamond subgraph of $G$.

    By minimality, let $f'$ be a $15$-D-coloring of $G-w_4$. The edge $w_4 y$ sees only the four edges incident to $y$ in $G-w_4$, so $\ell_{f'}(w_4 y)\ge 11$. Each of the other four edges incident to $w_4$ sees at most $11$ edges of $G-w_4$ by $G$, as in the earlier count for $uv$. Hence each has at least four available colors. By Lemma~\ref{extension}, $f'$ extends to a $15$-D-coloring of $G$, a contradiction.

    \textbf{Case 4.} There are three edges with both endpoints in $Z=\{w_1,\,w_2,\,w_3\}$, which means $\{w_1 w_2,\,w_1 w_3,\,w_2 w_3\}\subseteq E(G)$.

    In this case, we need to have $7-3=4$ edges each with one endpoint in $\{w_1,\,w_2,\,w_3\}$ and the other endpoint in $\{w_4,\,x_4\}$. However, as $G$ is a $5$-regular graph, and each of $w_1,\,w_2,\,w_3$ already has four neighbors ($w_1$ is adjacent to $u,\,v,\,w_2,\,w_3$; $w_2$ is adjacent to $u,\,v,\,w_1,\,w_3$; and $w_3$ is adjacent to $u,\,v,\,w_1,\,w_2$), there are at most three edges each with one endpoint in $\{w_1,\,w_2,\,w_3\}$ and the other endpoint in $\{w_4,\,x_4\}$, a contradiction.

    In each case, we get a contradiction. Hence, such a minimal counterexample $G$ does not exist, and the case $\Delta=5$ of Theorem~\ref{delta5} is established.
\end{proof}

\section{Remarks on planar graphs}

For B-colorings, Gy\'arf\'as et al.~\cite{GMRS} conjectured that, if $\Delta$ is sufficiently large, then $\chi'_B(G)\le 2\Delta$ for every planar graph $G$ with maximum degree $\Delta$. In~\cite{KWZ}, Kong et al.~verified this conjecture by proving that $\chi'_B(G)\le 2\Delta$ if $G$ is a planar graph with $\Delta\ge 38$. This upper bound is sharp, as it can be attained by $K_{2,\,\Delta}$.

As every B-coloring is a D-coloring, we also have $\chi'_D(G)\le 2\Delta$ if $G$ is a planar graph with $\Delta\ge 38$. Moreover, by Theorem~\ref{delta5}, we know that if $G$ is a planar graph with $\Delta\le 3$, then 

\begin{itemize}
    \item $\chi'_D(G)=1$ if $\Delta=1$;
    \item $\chi'_D(G)\le 3$ if $\Delta=2$, which is sharp because $K_3$ is planar and $\chi'_D(K_3)=3$;
    \item $\chi'_D(G)\le 6$ if $\Delta=3$, which is sharp because $K_4$ is planar and $\chi'_D(K_4)=6$.
\end{itemize}

For $\Delta\ge 4$, we propose the following conjecture.

\begin{conjecture}\label{conj3}
    For every planar graph $G$ with maximum degree $\Delta\ge 4$,
    \[
    \chi'_D(G)\le\begin{cases}
        9 &\text{if } \Delta=4, \\
        10 &\text{if } \Delta=5, \\
        2\Delta-1 &\text{if } \Delta\ge 6.
    \end{cases}
    \]
\end{conjecture}

As noted in Section~1, after the initial version of this paper became publicly available, Conjecture~\ref{conj3} was proven~\cite{HKW,Ve}. The graphs in Figure~\ref{sharp} show that all three bounds are sharp:

\begin{itemize}
    \item For $\Delta=4$, under a D-coloring of Figure~\ref{sharp}(a), all $9$ edges must get distinct colors.
    \item For $\Delta=5$, under a D-coloring of Figure~\ref{sharp}(b), the $10$ edges in $\{uv,\,uw_1,\,uw_2,\,uw_3,\,\allowbreak uw_4,\,vw_1,\,vw_2,\,vw_3,\,vw_4,\,w_2 w_3\}$ must get distinct colors (but $w_1 w_2$ may have the same color as $uw_4$ or $vw_4$, and $w_3 w_4$ may have the same color as $uw_1$ or $vw_1$).
    \item For $\Delta\ge 6$, under a D-coloring of Figure~\ref{sharp}(c), all $2\Delta-1$ edges must get distinct colors. Note that Figure~\ref{sharp}(c) is constructed by taking the complete bipartite graph $K_{2,\,\Delta-1}$ and connecting the two vertices in the part of size $2$.
\end{itemize}

\begin{figure}[H]
    \input{sharp}
\end{figure}

\noindent \textbf{Data Availability} There is no associated data.

\noindent \textbf{Declaration of generative AI use} During the preparation of this work, the author used GPT-6.2 Sol to assist with proofreading and refining the exposition. The author independently reviewed and edited all AI-generated content and takes full responsibility for the final manuscript.

\section*{Statements and Declarations}
\noindent \textbf{Competing Interests} The author has no competing interests to declare.


\begin{thebibliography}{10}

\bibitem{An}
L.~D. Andersen.
\newblock The strong chromatic index of a cubic graph is at most {$10$}.
\newblock {\em Discrete Math.}, 108(1-3):231--252, 1992.

\bibitem{BLMSS}
L.~Bezegov\'a, B.~Lu\v{z}ar, M.~Mockov\v{c}iakov\'a, R.~Sot\'ak, and R.~\v{S}krekovski.
\newblock Star edge coloring of some classes of graphs.
\newblock {\em J. Graph Theory}, 81(1):73--82, 2016.

\bibitem{BPP}
M.~Bonamy, T.~Perrett, and L.~Postle.
\newblock Colouring graphs with sparse neighbourhoods: bounds and applications.
\newblock {\em J. Combin. Theory Ser. B}, 155:278--317, 2022.

\bibitem{BJ}
H.~Bruhn and F.~Joos.
\newblock A stronger bound for the strong chromatic index.
\newblock {\em Combin. Probab. Comput.}, 27(1):21--43, 2018.

\bibitem{CHYZ}
L.~Chen, M.~Huang, G.~Yu, and X.~Zhou.
\newblock The strong edge-coloring for graphs with small edge weight.
\newblock {\em Discrete Math.}, 343(4):111779, 11, 2020.

\bibitem{DHDKV}
E.~Davey, E.~Hurley, R.~de~Joannis~de Verclos, R.~J. Kang, and J.~Volec.
\newblock Strong edge-colouring via local flag algebras.
\newblock {\em arXiv preprint arXiv:2607.17421}, 2026.

\bibitem{DMS}
Z.~Dvo\v{r}\'ak, B.~Mohar, and R.~\v{S}\'amal.
\newblock Star chromatic index.
\newblock {\em J. Graph Theory}, 72(3):313--326, 2013.

\bibitem{EG}
P.~Erd\H{o}s and A.~Gy\'arf\'as.
\newblock A variant of the classical {R}amsey problem.
\newblock {\em Combinatorica}, 17(4):459--467, 1997.

\bibitem{EN}
P.~Erd\H{o}s and J.~Ne\v{s}et\v{r}il.
\newblock {\em Irregularities of partitions}, volume~8 of {\em Algorithms and Combinatorics: Study and Research Texts, edited by G. Hal\'asz and V.T. S\'os}.
\newblock Springer-Verlag, Berlin, 1989.
\newblock Papers from the meeting held in Fert\H od, July 7--11, 1986.

\bibitem{GMRS}
A.~Gy\'arf\'as, R.~R. Martin, M.~Ruszink\'o, and G.~N. S\'ark\"ozy.
\newblock Proper edge colorings of planar graphs with rainbow {$C_4$}-s.
\newblock {\em J. Graph Theory}, 107(4):833--846, 2024.

\bibitem{GS}
A.~Gy\'arf\'as and G.~N. S\'ark\"ozy.
\newblock ``{L}ess'' strong chromatic indices and the {$(7, 4)$}-conjecture.
\newblock {\em Studia Sci. Math. Hungar.}, 60(2-3):109--122, 2023.

\bibitem{GSW}
A.~Gy\'arf\'as, G.~N. S\'ark\"ozy, and A.~Z. Wagner.
\newblock Perfect proper edge colorings of regular bipartite graphs with rainbow {$C_4$}-s.
\newblock {\em Discrete Math.}, 349(6):Paper No. 115005, 12, 2026.

\bibitem{HHT}
P.~Hor\'ak, Q.~He, and W.~T. Trotter.
\newblock Induced matchings in cubic graphs.
\newblock {\em J. Graph Theory}, 17(2):151--160, 1993.

\bibitem{HKW}
X.~Hu, J.~Kong, and Y.~Wang.
\newblock D-coloring of planar graphs.
\newblock {\em arXiv preprint arXiv:2607.14837}, 2026.

\bibitem{HSY}
M.~Huang, M.~Santana, and G.~Yu.
\newblock Strong chromatic index of graphs with maximum degree four.
\newblock {\em Electron. J. Combin.}, 25(3):Paper No. 3.31, 24, 2018.

\bibitem{HDK}
E.~Hurley, R.~de~Joannis~de Verclos, and R.~J. Kang.
\newblock An improved procedure for colouring graphs of bounded local density.
\newblock {\em Adv. Comb.}, pages Paper No. 7, 33, 2022.

\bibitem{KWZ}
J.~Kong, Y.~Wang, and M.~Zheng.
\newblock B-coloring of planar graphs.
\newblock {\em Journal of Graph Theory}, published online, 2026.

\bibitem{LL}
Y.~Lin and W.~Lin.
\newblock Strong chromatic index of claw-free graphs with edge weight seven.
\newblock {\em Discuss. Math. Graph Theory}, 44(4):1311--1325, 2024.

\bibitem{LLH}
J.~Lu, H.~Liu, and X.~Hu.
\newblock On strong edge-coloring of graphs with maximum degree 5.
\newblock {\em Discrete Appl. Math.}, 344:120--128, 2024.

\bibitem{LLY}
J.-B. Lv, X.~Li, and G.~Yu.
\newblock On strong edge-coloring of graphs with maximum degree 4.
\newblock {\em Discrete Appl. Math.}, 235:142--153, 2018.

\bibitem{MR}
M.~Molloy and B.~Reed.
\newblock A bound on the strong chromatic index of a graph.
\newblock {\em J. Combin. Theory Ser. B}, 69(2):103--109, 1997.

\bibitem{NN}
K.~Nakprasit and K.~Nakprasit.
\newblock The strong chromatic index of graphs with restricted {O}re-degrees.
\newblock {\em Ars Combin.}, 118:373--380, 2015.

\bibitem{NY}
S.~Nelson and G.~Yu.
\newblock Planar graphs with {O}re-degree at most seven is strongly 13-edge-colorable.
\newblock {\em Discrete Appl. Math.}, 386:345--364, 2026.

\bibitem{TW}
L.~Tian and R.~Wang.
\newblock A stronger upper bound on the {D}-chromatic index.
\newblock {\em arXiv preprint arXiv:2609.01875}, 2026.

\bibitem{Ve}
H.~Verma.
\newblock A proof of {W}ang's planar {D}-coloring conjecture.
\newblock {\em arXiv preprint arXiv:2610.33322}, 2026.

\bibitem{Wa1}
R.~Wang.
\newblock An improved bound on the strong chromatic index via local flag algebras.
\newblock {\em arXiv preprint arXiv:2611.22031}, 2026.

\bibitem{Wa2}
R.~Wang.
\newblock Strong edge-coloring of graphs with maximum edge weight seven.
\newblock {\em J. Comb. Optim.}, 51(1):Paper No. 2, 11, 2026.

\bibitem{Wa3}
R.~Wang.
\newblock Strong edge-coloring of sparse graphs with ore-degree 7 or 8.
\newblock {\em arXiv preprint arXiv:2602.03862}, 2026.

\bibitem{WL}
J.~Wu and W.~Lin.
\newblock The strong chromatic index of a class of graphs.
\newblock {\em Discrete Math.}, 308(24):6254--6261, 2008.

\end{thebibliography}
\end{document}